\documentclass[12pt]{article}%
\usepackage{amsmath}
\usepackage{amssymb}
\usepackage{amsfonts}%
\usepackage{graphicx}
\providecommand{\U}[1]{\protect\rule{.1in}{.1in}}
\newtheorem{theorem}{Theorem}[section]

\newtheorem{definition}{Definition}[section]

\newtheorem{proposition}{Proposition}
\newtheorem{remark}{Remark}[section]

\newenvironment{proof}[1][Proof]{\noindent\textbf{#1} }{\ \rule{0.5em}{0.5em}}
\begin{document}

\title{The Foundations of Analysis}
\author{Larry Clifton\\larry@cliftonlabs.net}
\date{March 22, 2013}
\maketitle
\tableofcontents

\section{Introduction}

Here we introduce the real numbers. On the one hand, this is a modern
introduction based on morphisms between objects in an algebraic category. On
the other hand, it is an ancient introduction with 24 of the theorems dating
back to about 300 B.C. It is hoped that the complete and elementary nature of
this work will show that it is practical to introduce the real numbers from a
categorical perspective to students who may never study abstract algebra, but
will use real numbers on a regular basis in their future course work and
professional careers.

In this endeavor we have been inspired by Landau's little book
\emph{Grundlagen der Analysis}\cite{Landau little book} from which we have
appropriated the title of the present work.

\section{Magnitude Spaces}

\begin{definition}
\label{definition magnitude}A \textbf{magnitude space} is a nonempty set $M$
together with a binary operation on $M$, which we usually denote by $+$, such
that for any $a,b,c\in M$

(i)(associativity) $a+\left(  b+c\right)  =\left(  a+b\right)  +c$,

(ii)(commutativity) $a+b=b+a$, and

(iii)(trichotomy) exactly one of the following is true: $a=b+d$ for some

$d\in M$, or $a=b$, or $b=a+d$ for some $d\in M$.
\end{definition}

\begin{definition}
[The whole is greater than the part.]\label{definition <}If $a$ and $b$ are
elements of a magnitude space $M$ and $b=a+d$ for some $d\in M$, we say $a$ is
\textbf{less than} $b$ and write $a<b$ or equivalently we say $b$ is
\textbf{greater than} $a$ and write $b>a$.
\end{definition}

\begin{remark}
Since $a$ and $b$ are smaller than $a+b$ and hence not equal to $a+b$, a
magnitude space does not have an additive identity or zero element.
\end{remark}

The concept of magnitude spaces has a long history. That the concept was
recognized as a formal abstraction before 300 B.C. is evidenced by the
following two quotes. The first is Proposition 16 from Book V of Euclid's
\emph{Elements}.

\begin{quotation}
\textbf{Given--}Four proportionate magnitudes $a$, $b$, $c$, $d$ with $a$ to
$b$ the same as $c$ to $d$.

\textbf{To be Shown--}The alternates will also be a proportion, $a$ to $c$ the
same as $b$ to $d$.
\end{quotation}

The second, dating from a generation or two before Euclid, is a comment by
Aristotle (Posterior Analytics I, 5, 20).

\begin{quotation}
Alternation used to be demonstrated separately of numbers, lines, solids, and
time intervals, though it could have been proved of them all by a single
demonstration. Because there was no single name to denote that in which
numbers, lines, solids, and time intervals are identical, and because they
differed specifically from one another, this property was proved of each of
them separately. Today, however, the proof is commensurately universal, for
they do not possess this attribute as numbers, lines, solids, and time
intervals, but as manifesting this generic character which they are postulated
as possessing universally.
\end{quotation}

We do not have an explicit list of the properties which the various kinds of
magnitudes are postulated as possessing universally from this time period. All
we have are a number of \textquotedblleft Common Notions\textquotedblright%
\ such as \textquotedblleft the whole is greater than the
part\textquotedblright\ and \textquotedblleft if equals are subtracted from
equals, the remainders will be equal.\textquotedblright\ The only complete
work on proportions surviving from this period is Book V of Euclid's
\emph{Elements}. Here we find rigorous reasoning without an explicit
foundation. We do not know if there was an explicit foundation generally known
at the time. What we do know is that, starting with the definition of
magnitude spaces above, we can establish a foundation by which the
propositions in Book V of Euclid can be proved in accordance with present day
standards of rigor.\cite{Stein} And from the propositions in Book V of Euclid
there follows the theory establishing the real numbers in the modern
categorical sense.

\section{Functions}

Here we review function terminology and a few basic theorems.

\begin{definition}
To indicate that $\varphi$ is a function from a set $S$ into a set $S^{\prime
}$, we write $\varphi:S\rightarrow S^{\prime}$. We also refer to functions as
mappings and if $\varphi a=b$ we say that $\varphi$ \textbf{maps} $a$ into $b$.
\end{definition}

\begin{definition}
\label{definition equal functions}Two functions $\varphi$ and $\psi$ from a
set $S$ into a set $S^{\prime}$ are \textbf{equal} if for all $a\in S$,
$\varphi a=\psi a$. And in this case we write $\varphi=\psi$.
\end{definition}

\begin{definition}
\label{definition one-to-one}A function $\varphi:S\rightarrow S^{\prime}$ is
\textbf{one-to-one} if for all $a,b\in S$, $a\neq b$ implies $\varphi
a\neq\varphi b$.
\end{definition}

\begin{definition}
\label{definition onto}A function $\varphi:S\rightarrow S^{\prime}$ is
\textbf{onto} if for each $a^{\prime}\in S^{\prime}$ there is some $a\in S$
such that $\varphi a=a^{\prime}$. And in this case we say that $\varphi$ maps
$S$ \textbf{onto} $S^{\prime}.$
\end{definition}

\begin{definition}
\label{definition composition}If $\varphi_{1}:S_{1}\rightarrow S_{2}$ and
$\ \chi:S_{2}\rightarrow S_{3}$ are functions, then we define the
\textbf{composition} $\chi\circ\varphi:S_{1}\rightarrow S_{3}$ by $\left(
\chi\circ\varphi\right)  a=\chi\left(  \varphi a\right)  $ for all $a\in
S_{1}$. In most cases, we omit the $\circ$ symbol between the functions and
define the composition $\chi\varphi:S_{1}\rightarrow S_{3}$ by $\left(
\chi\varphi\right)  a=\chi\left(  \varphi a\right)  $
\end{definition}

\begin{definition}
\label{definition identity function}The \textbf{identity function}
$i_{S}:S\rightarrow S$ maps every element of a set $S$ to itself. I.e. for all
$a\in S$, $i_{S}a=a$.
\end{definition}

\begin{theorem}
\label{composition is associative}Composition of functions is associative
\end{theorem}

\begin{proof}
Assume $S,S^{\prime},S^{\prime\prime}$ and $\varphi:S\rightarrow S^{\prime}$,
$\chi:S^{\prime}\rightarrow S^{\prime\prime}$, $\psi:S^{\prime\prime
}\rightarrow S^{\prime\prime\prime}$ are functions. If $a\in S$ then%
\begin{align*}
\left(  \psi\left(  \chi\varphi\right)  \right)  a  & =\psi\left(  \left(
\chi\varphi\right)  a\right)  \text{ (Definition \ref{definition composition}%
)}\\
& =\psi\left(  \chi\left(  \varphi a\right)  \right)  \text{ (Definition
\ref{definition composition})}\\
& =\left(  \psi\chi\right)  \left(  \varphi a\right)  \text{ (Definition
\ref{definition composition})}\\
& =\left(  \left(  \psi\chi\right)  \varphi\right)  a\text{ (Definition
\ref{definition composition})}%
\end{align*}
and therefore $\psi\left(  \chi\varphi\right)  =\left(  \psi\chi\right)
\varphi$ by Definition \ref{definition equal functions}.
\end{proof}

\begin{theorem}
\label{identity 1-1 and onto}The identity function $i_{S}:S\rightarrow S$ is
one-to-one and onto.
\end{theorem}

\begin{proof}
Suppose $a,b\in S$ and $a\neq b$. Now $i_{S}a=a$ and $i_{S}b=b$ by Definition
\ref{definition identity function} and hence $i_{S}a\neq i_{S}b$. Therefore
$i_{S}$ is one-to one according to Definition \ref{definition one-to-one}.

And for any $a\in S$ there is some $c\in S$ (namely $c=a$) such that
$i_{S}c=a$ by Definition \ref{definition identity function}. Therefore $i_{S}$
is onto according to Definition \ref{definition onto}.
\end{proof}

\begin{theorem}
\label{proving function is onto}If $\varphi:S\rightarrow S^{\prime}$,
$\psi:S^{\prime}\rightarrow S$, and $\psi\varphi=i_{S}$, then $\psi$ is onto.
\end{theorem}

\begin{proof}
For any $a\in S$%
\begin{align*}
a  & =i_{S}a\text{ (Definition \ref{definition identity function})}\\
& =\left(  \psi\varphi\right)  a\text{ (Definition
\ref{definition equal functions})}\\
& =\psi\left(  \varphi a\right)  \text{. (Definition
\ref{definition composition})}%
\end{align*}
Thus there is an element in $S^{\prime}$ (namely $\varphi a$) which $\psi$
maps into $a$. Therefore $\psi$ is onto according to Definition
\ref{definition onto}.
\end{proof}

\begin{remark}
In the preceding theorem, we could also conclude that $\varphi$ is one-to-one.
\end{remark}

\section{Basic Equalities and Inequalities}

\begin{definition}
\label{definition trichotomous}A binary relation $<$ on a set $S$ is
\textbf{trichotomous} if for all $a,b\in S$, exactly one of the following is
true: $b<a$, or $a=b$, or $a<b$.
\end{definition}

\begin{definition}
\label{definition transitive}A binary relation $<$ on a set $S$ is
\textbf{transitive} if for all $a,b,c\in S$, $a<b$ and $b<c$ implies $a<c$.
\end{definition}

\begin{definition}
\label{definition strict linear order}A binary relation is a \textbf{strict
linear order}\ if it is trichotomous and transitive.
\end{definition}

\begin{definition}
Binary relations $<$ and $>$ on a set $S$ are \textbf{inverses} (to each
other) if for all $a,b\in S$, $a<b\,\ $if and only if $b>a$.
\end{definition}

\begin{theorem}
\label{order preserving functions}If $S$ and $S^{\prime}$ are sets with
trichotomous relations $<$ (and inverse relations $>$) and $\varphi
:S\rightarrow S^{\prime}$ is a function such that for all $a,b\in S$, $a<b$
implies $\varphi a<\varphi b$, then for all $a,b\in S$, $\varphi a$ has to
$\varphi b$ the same relation (%
$<$%
, =, or
$>$%
) as $a$ has to $b$, and $\varphi$ is one-to-one.
\end{theorem}

\begin{proof}
For $a,b\in S$ the three mutually exclusive cases%
\[
a<b\text{, or }a=b\text{, or }a>b
\]
imply, by assumption, the three mutually exclusive cases%
\[
\varphi a<\varphi b\text{, or }\varphi a=\varphi b\text{, or }\varphi
a>\varphi b
\]
respectively. The three converse implications follow from trichotomy. For
instance, assume $\varphi a<\varphi b$. If $a=b$, then $\varphi a=\varphi b$
which contradicts the assumption. If $a>b$, then $\varphi a>\varphi b$ which
also contradicts the assumption. Since $a=b$ and $a>b$ are incompatible with
our assumption, $a<b$ by trichotomy. Thus $\varphi a<\varphi b$ implies $a<b$.

I say $\varphi$ is one-to-one. For if $a,b\in S$ and $a\neq b$, then $a<b$ or
$a>b$ by trichotomy and hence $\varphi a<\varphi b$ or $\varphi b<\varphi a$
and hence $\varphi a\neq\varphi b$ by trichotomy. We have now shown that if
$a\neq b$ then $\varphi a\neq\varphi b$. Therefore $\varphi$ is one-to-one
according to Definition \ref{definition one-to-one}.
\end{proof}

\ 

In the remainder of this section lower case variables $a,$ $b,$ $c,$ and $d$
are elements of a magnitude space $M$.

\begin{theorem}
\label{< is trichotomous}The $<$ relation in a magnitude space is trichotomous.
\end{theorem}

\begin{proof}
From Definition \ref{definition magnitude}, exactly one of the following is
true: $a=b+d$ for some $d\in M$, $a=b$, or $b=a+d$ for some $d\in M$. From
Definition \ref{definition <}, exactly one of the following is true: $b<a$, or
$a=b$, or $a<b$. Therefore$\ $the $<$ relation in a magnitude space is
trichotomous according to Definition \ref{definition trichotomous}.
\end{proof}

\begin{theorem}
[Translation Invariance]\label{Translation by a}If $b<c$, then $a+b<a+c$.
\end{theorem}

\begin{proof}
If $b<c$, then $c=b+d$ for some $d$ by Definition \ref{definition <}. Hence
\[
a+c=a+\left(  b+d\right)  =\left(  a+b\right)  +d
\]
by Definition \ref{definition magnitude} (associativity) and therefore
$a+b<a+c$ according to Definition \ref{definition <}.
\end{proof}

\begin{theorem}
[If equals are added to equals or unequals ...]%
\label{equals added to equals or unequals}$a+b$ has to $a+c$ the same relation
(%
$<$%
, =, or
$>$%
) as $b$ has to $c$.
\end{theorem}

\begin{proof}
Fix $a$ and define a function $\varphi:M\rightarrow M$ by $\varphi b=a+b$. If
$b<c$, then $\varphi b<\varphi c$ by the preceding theorem. Therefore $\varphi
b=a+b$ has to $\varphi c=a+c$ the same relation (%
$<$%
, =, or
$>$%
) as $b$ has to $c$ by Theorem \ref{order preserving functions}.
\end{proof}

\begin{remark}
If $a<b$, then, from Definition \ref{definition <}, there is some $d$ such
that $b=a+d$. In fact there is only one such element. For if $a+d=a+d^{\prime
}$, then $d=d^{\prime}$ by the preceding theorem.
\end{remark}

\begin{definition}
\label{definition subtraction}If $a<b$, we define $b-a$ to be the unique
element $d$ such that $b=a+d$.
\end{definition}

\begin{theorem}
\label{b-a < a}If $a<b$, then $b-a<b$.
\end{theorem}

\begin{proof}
Definitions \ref{definition <} and \ref{definition subtraction}
\end{proof}

\begin{remark}
If $a<b$, then $\left(  b-a\right)  +a=a+\left(  b-a\right)  =b$ since
$\left(  b-a\right)  $ is, by definition, the unique element which when added
to $a$ yields $b$.
\end{remark}

\begin{theorem}
[Transitivity of $<$]\label{Transitivity of >}If $a<b$ and $b<c$, then $a<c$
and $c-a=\left(  c-b\right)  +\left(  b-a\right)  $.
\end{theorem}

\begin{proof}
If $a<b$ and $b<c$, then $c=\left(  c-b\right)  +b=\left(  c-b\right)
+\left(  b-a\right)  +a$ by Definition \ref{definition subtraction}. Hence
$a<c$ by Definition \ref{definition <} and $c-a=\left(  c-b\right)  +\left(
b-a\right)  $ according to Definition \ref{definition subtraction}.
\end{proof}

\begin{theorem}
The $<$ relation in a magnitude space is a strict linear order.
\end{theorem}

\begin{proof}
The $<$ relation in a magnitude space is trichotomous by Theorem
\ref{< is trichotomous}. And it is transitive by the preceding theorem and
Definition \ref{definition transitive}. Therefore the $<$ relation in a
magnitude space is a strict linear order according to Definition
\ref{definition strict linear order}.
\end{proof}

\begin{theorem}
[If equals are subtracted from equals or unequals ...]%
\label{equals subtracted from equals or unequals}If $b>a$ and $c>a$, then
$b-a$ has to $c-a$ the same relation (%
$<$%
, =, or
$>$%
) as $b$ has to $c$.
\end{theorem}

\begin{proof}
$a+\left(  b-a\right)  $ has to $a+\left(  c-a\right)  $ the same relation (%
$<$%
, =, or
$>$%
) as $\left(  b-a\right)  $ has to $\left(  c-a\right)  $ by Theorem
\ref{equals added to equals or unequals}. But $a+\left(  b-a\right)  =b$ and
$a+\left(  c-a\right)  =c$ by Definition \ref{definition subtraction}.
Therefore $b-a$ has to $c-a$ the same relation (%
$<$%
, =, or
$>$%
) as $b$ has to $c$.
\end{proof}

\begin{theorem}
\label{moving b accross <=>}If $a>b$, then $a$ has to $b+c$ the same relation
(%
$<$%
, =, or
$>$%
) as $a-b$ has to $c$.
\end{theorem}

\begin{proof}
$b+\left(  a-b\right)  $ has to $b+c$ the same relation (%
$<$%
, =, or
$>$%
) as $a-b$ has to $c$ by Theorem \ref{equals added to equals or unequals}. And
$a=b+\left(  a-b\right)  $ by Definition \ref{definition subtraction}.
Therefore $a$ has to $b+c$ the same relation (%
$<$%
, =, or
$>$%
) as $a-b$ has to $c$.
\end{proof}

\section{Magnitude Space Embeddings}

In this section $M$, $M^{\prime}$, and $M^{\prime\prime}$ are magnitude spaces.

\begin{definition}
\label{definition homomorphism}A mapping $\varphi:M\rightarrow M^{\prime}$ is
a \textbf{homomorphism} if $\varphi\left(  a+b\right)  =\varphi a+\varphi b $
for all $a,b\in M$. A homomorphism which is one-to-one is an
\textbf{embedding}.
\end{definition}

\begin{definition}
\label{definition isomorphism isomorphic}An embedding $\varphi:M\rightarrow
M^{\prime}$ which is onto as a mapping is an \textbf{isomorphism}. If there is
an isomorphism from $M$ onto $M^{\prime}$ we say $M$ is \textbf{isomorphic} to
$M^{\prime}$.
\end{definition}

\begin{definition}
\label{definition endomorphism automorphism}A homomorphism $\varphi
:M\rightarrow M$ of a magnitude space into itself is an \textbf{endomorphism}
and an endomorphism which is one-to-one and onto as a map is an
\textbf{automorphism}.
\end{definition}

\begin{definition}
\label{definition sum of embeddings}If $\varphi:M\rightarrow M^{\prime}$ and
$\chi:M\rightarrow M^{\prime}$ are two functions, their \textbf{sum} is the
function $\left(  \varphi+\chi\right)  :M\rightarrow M^{\prime}$ defined by
$\left(  \varphi+\chi\right)  a=\varphi a+\chi a$ for all $a\in M$.
\end{definition}

The next two theorems show that homomorphisms between magnitude spaces are
always embeddings.

\begin{theorem}
\label{morphisms monotonic}If $\varphi:M\rightarrow M^{\prime}$ is a
homomorphism and $a<b$, then\newline$\varphi a<\varphi b$ and $\varphi
b-\varphi a=\varphi\left(  b-a\right)  .$
\end{theorem}

\begin{proof}
If $\varphi$ is a homomorphism and $a<b$, then%
\begin{align*}
\varphi b  & =\varphi\left(  a+\left(  b-a\right)  \right)  \text{ (Definition
\ref{definition subtraction})}\\
& =\varphi a+\varphi\left(  b-a\right)  \text{. (Definition
\ref{definition homomorphism})}%
\end{align*}
Hence $\varphi a<\varphi b$ according to Definition \ref{definition <} and
$\varphi b-\varphi a=\varphi\left(  b-a\right)  $ according to Definition
\ref{definition subtraction}.
\end{proof}

\begin{theorem}
\label{morphisms 1-1}If$\ \varphi:M\rightarrow M^{\prime}$ is a homomorphism,
then $\varphi a$ has to $\varphi b$ the same relation (%
$<$%
, =, or
$>$%
) as $a$ has to $b$ and $\varphi$ is an embedding.
\end{theorem}

\begin{proof}
From the preceding theorem, $a<b$ implies $\varphi a<\varphi b$. Hence
$\varphi a$ has to $\varphi b$ the same relation (%
$<$%
, =, or
$>$%
) as $a$ has to $b$ and $\varphi$ is one-to-one by Theorem
\ref{order preserving functions}. And since $\varphi$ is a homomorphism and is
one-to-one, $\varphi$ is an embedding according to Definition
\ref{definition homomorphism}.
\end{proof}

\begin{theorem}
\label{identity homomorphism}The identity function $i_{M}$ in a magnitude
space $M$ is an automorphism.
\end{theorem}

\begin{proof}
For any $a,b\in M$, $i_{M}\left(  a+b\right)  =a+b=i_{M}a+i_{M}b$ by
Definition \ref{definition identity function}. Hence $i_{M}$ is a homomorphism
according to Definition \ref{definition homomorphism}. And $i_{M}$ is
one-to-one and onto by Theorem \ref{identity 1-1 and onto}. Therefore $i_{M}$
is an automorphism according to Definitions \ref{definition homomorphism} and
\ref{definition endomorphism automorphism}.
\end{proof}

\begin{theorem}
\label{sum of homomorphisms}The sum of two embeddings is an embedding.
\end{theorem}

\begin{proof}
If $\varphi:M\rightarrow M^{\prime}$ and $\chi:M\rightarrow M^{\prime}$ are
embeddings and $a,b\in M$, then
\begin{align*}
\left(  \varphi+\chi\right)  \left(  a+b\right)   & =\varphi\left(
a+b\right)  +\chi\left(  a+b\right)  \text{ (Definition
\ref{definition sum of embeddings})}\\
& =\left(  \varphi a+\varphi b\right)  +\left(  \chi a+\chi b\right)  \text{
(Definition \ref{definition homomorphism})}\\
& =\left(  \varphi a+\chi a\right)  +\left(  \varphi b+\chi b\right)  \text{
(commutativity and associativity of +)}\\
& =\left(  \varphi+\chi\right)  a+\left(  \varphi+\chi\right)  b\text{.
(Definition \ref{definition sum of embeddings})}%
\end{align*}
Therefore $\varphi+\chi$ is an embedding by Definition
\ref{definition homomorphism} and Theorem \ref{morphisms 1-1}.
\end{proof}

\begin{theorem}
\label{composition of homomophisms}The composition of two embeddings is an embedding.
\end{theorem}

\begin{proof}
If $\varphi:M\rightarrow M^{\prime}$ and $\chi:M^{\prime}\rightarrow
M^{\prime\prime}$ are embeddings and $a,b\in M$, then%
\begin{align*}
\left(  \chi\varphi\right)  \left(  a+b\right)   & =\chi\left(  \varphi\left(
a+b\right)  \right)  \text{ (Definition \ref{definition composition})}\\
& =\chi\left(  \varphi a+\varphi b\right)  \text{ (Definition
\ref{definition homomorphism})}\\
& =\chi\left(  \varphi a\right)  +\chi\left(  \varphi b\right)  \text{
(Definition \ref{definition homomorphism})}\\
& =\left(  \chi\varphi\right)  a+\left(  \chi\varphi\right)  b\text{.
(Definition \ref{definition composition})}%
\end{align*}
Therefore $\varphi\chi$ is an embedding by Definition
\ref{definition homomorphism} and Theorem \ref{morphisms 1-1}.
\end{proof}

\section{Classification of Magnitude Spaces}

\begin{definition}
\label{definition <= and >=}Let $<$ be a strict linear order with inverse $>$.
By $a\leq b$ we shall mean $a<b$ or $a=b$. By $a\geq b$ we shall mean $a>b$ or
$a=b$.
\end{definition}

\begin{definition}
\label{definition bounds}Let $S$ be a set with a strict linear order $<$ and
let $A$ be a nonempty subset of $S$. We say $b\in S$ is a \textbf{lower bound}
of $A$ if $b\leq a$ for every $a\in A$. We say that $b$ is a \textbf{smallest}
or \textbf{least} element of $A$ if $b$ is a lower bound of $A$ and $b\in A$.
We say $b\in S$ is an \textbf{upper bound} of $A$ if $a\leq b$ for every $a\in
A$. We say that $b$ is a \textbf{largest} or \textbf{greatest} element of $A$
if $b$ is an upper bound of $A$ and $b\in A$.
\end{definition}

\begin{definition}
\label{definition discrete}A magnitude space is \textbf{discrete} if it has a
smallest element; otherwise it is \textbf{nondiscrete}.
\end{definition}

\begin{definition}
\label{definition well ordered}A magnitude space $M$ is \textbf{well ordered}
if every nonempty subset of $M$ has a smallest element.
\end{definition}

\begin{definition}
\label{definition complete}A magnitude space is \textbf{complete} if every
nonempty subset with an upper bound has a least upper bound or, in other
words, the set of upper bounds has a least element.
\end{definition}

\begin{definition}
\label{definition continuous}A magnitude space is \textbf{continuous} if it is
complete and nondiscrete.
\end{definition}

\begin{definition}
[H\"{o}lder]\label{definition Archimedean}A magnitude space $M$ is
\textbf{Archimedean} if for any element $a\in M$ and any nonempty subset $A$
with an upper bound, there is some element $\zeta\in M$ such that $\zeta\in A$
and $\zeta+a\notin A$.
\end{definition}

\begin{remark}
The property of being Archimedean is usually defined in terms of integral
multiples.\cite{Holder}
\end{remark}

\begin{theorem}
\label{Transitivity mixed}If $a\leq b$ and $b<c$, then $a<c$. And if $a<b$ and
$b\leq c$, then $a<c$.
\end{theorem}

\begin{proof}
Assume $a\leq b$ and $b<c$. Then $a<b$ or $a=b$ by Definition
\ref{definition <= and >=}. If $a<b$ then since also $b<c$, $a<c$ by Theorem
\ref{Transitivity of >}. If $a=b$ then since also $b<c$, $a<c$. Thus in both
cases $a<c$. The second part of the theorem follows by similar reasoning.
\end{proof}

\begin{theorem}
\label{element greater than upper bound}If an element is greater than an upper
bound of a set, then it is an upper bound of the set but not an element of the set.
\end{theorem}

\begin{proof}
Let $S$ be a linearly ordered set, $A$ a nonempty subset of $S$, $a$ an upper
bound of $A$, and $a<b$. I say $b$ is an upper bound of $A$ and $b\notin A$.

If $c\in A$, then $c\leq a$ by Definition \ref{definition bounds}. And $a<b$
by assumption. Thus $c\leq b$ by Theorem \ref{Transitivity mixed} and
Definition \ref{definition <= and >=}. We have now shown that if $c\in A$,
then $c\leq b$. Therefore $b$ is an upper bound of $A$.

Now suppose $b\in A$. Then $b\leq a$ by Definition \ref{definition bounds}$.$
But this contradicts the assumption $a<b$ by Definition
\ref{definition <= and >=} and Theorem \ref{< is trichotomous}. Therefore
$b\notin A$.
\end{proof}

\begin{theorem}
\label{element less than lower bound}If an element is less than a lower bound
of a set, then it is a lower bound of the set but not an element of the set.
\end{theorem}

\begin{proof}
Similar to proof of previous theorem.
\end{proof}

\begin{theorem}
\label{discrete gap}If $M$ is a discrete magnitude space with smallest element
$a$, and $b\in M$, then there is no $c\in M$ such that $b<c<b+a$.
\end{theorem}

\begin{proof}
Suppose there is such a $c$. Then $\left(  b+a\right)  -b=\left(  \left(
b+a\right)  -c\right)  +\left(  c-b\right)  $ by Theorem
\ref{Transitivity of >}. But $\left(  b+a\right)  -b=a$ by Definition
\ref{definition subtraction}. Hence $c-b<a$ by Definition \ref{definition <}.
But $a$ is the smallest element of $M$ by assumption and so we have a contradiction.
\end{proof}

\begin{theorem}
\label{well ordered implies discrete and complete}A well ordered magnitude
space is discrete and complete.
\end{theorem}

\begin{proof}
Let $M$ be a well ordered magnitude space. Since $M$ is a subset of $M$, $M$
has a smallest element by Definition \ref{definition well ordered}. Therefore
$M$ is discrete according to Definition \ref{definition discrete}.

Now let $A$ be any nonempty subset of $M$ with an upper bound. If $B$ is the
set of upper bounds of $A$, then $B$ is a nonempty subset of $M$ and hence $B
$ has a least element by Definition \ref{definition well ordered}. Therefore
$M$ is complete according to Definition \ref{definition complete}.
\end{proof}

\begin{theorem}
[H\"{o}lder]\label{Holder}A complete magnitude space is Archimedean.
\end{theorem}

\begin{proof}
Let $M$ be a complete magnitude space, $a\in M$, and $A$ a nonempty subset of
$M$ with an upper bound. Since $M$ is complete, $A$ has a least upper bound
$\zeta^{\prime}$ by Definition \ref{definition complete}.

Case 1: $\zeta^{\prime}\leq a$. Since Let $\zeta$ be any element of $A$. Then
$a<\zeta+a$ by Definition \ref{definition <} and so $\zeta^{\prime}<\zeta+a$
by Theorem \ref{Transitivity mixed}$.$ But then $\zeta+a$ is greater than an
upper bound of $A$ and hence $\zeta+a\notin A$ by Theorem
\ref{element greater than upper bound}.

Case 2: $a<\zeta^{\prime}$. Since $\zeta^{\prime}-a<\zeta^{\prime}$ by Theorem
\ref{b-a < a}, $\zeta^{\prime}-a$ is less than the least upper bound of $A$
and hence is not an upper bound of $A$ by Theorem
\ref{element less than lower bound}. Thus there is some $\zeta\in A$ such that
$\zeta^{\prime}-a<\zeta$ by Definition \ref{definition bounds}. And
$\zeta^{\prime}<\zeta+a$ by Theorem \ref{moving b accross <=>} and so
$\zeta+a$ is greater than an upper bound of $A$ and hence $\zeta+a\notin A$ by
Theorem \ref{element greater than upper bound}.

We have shown, in both cases, that there is some $\zeta\in M$ such that
$\zeta\in A$ and $\zeta+a\notin A$. Therefore $M$ is Archimedean according to
Definition \ref{definition Archimedean}.
\end{proof}

\begin{theorem}
A discrete Archimedean magnitude space is well ordered.
\end{theorem}

\begin{proof}
Assume $M$ is a discrete Archimedean magnitude space with smallest element $a
$ and $A$ is a nonempty subset of $M$. I say $A$ has a smallest element.

Let $B$ be the set of all lower bounds of $A$. If $b\in A$, then $a\leq b$ by
Definition \ref{definition bounds}. Therefore $a\in B$ and so $B$ is nonempty.
Now let $c$ be any fixed element of $A$. If $b\in B$, then $b\leq c $. Thus
$c$ is an upper bound of $B$ according to Definition \ref{definition bounds}
and so $B$ is a nonempty subset of $M$ with an upper bound.

But $M$ is Archimedean and hence there is some element $\zeta\in M$ such that
$\zeta\in B$ and $\zeta+a\notin B$ by Definition \ref{definition Archimedean}.
Or, in other words, $\zeta$ is a lower bound of $A$ and $\zeta+a$ is not a
lower bound of $A$. And because $\zeta+a$ is not a lower bound of $A$, there
is some $b\in A$ such that $b<\zeta+a$ by Definitions \ref{definition bounds}
and \ref{definition <= and >=} and trichotomy. And since $\zeta$ is a lower
bound of $A$, $\zeta\leq b$ by Definition \ref{definition bounds}. But
$\zeta<b$ and $b<\zeta+a$ is impossible by Theorem \ref{discrete gap}. Thus
$\zeta=b$ by Definition \ref{definition <= and >=}. Therefore $\zeta\in A$ and
$\zeta$ is a lower bound of $A$ and hence $\zeta$ is the smallest element of
$A$ according to Definition \ref{definition bounds}.
\end{proof}

\section{Well Ordered Magnitude Spaces}

We next prove a form of mathematical induction for well ordered magnitude spaces.

\begin{theorem}
\label{induction}If $M$ is a well ordered magnitude space with smallest
element $a$, and $A$ is subset of $M$ containing $a$ such that $c\in A$
implies $c+a\in A$, then $A=M$.
\end{theorem}

\begin{proof}
Suppose $A\neq M$. Then the set $B$ consisting of those elements in $M$ which
are not in $A$ is nonempty and hence has a smallest element by Definition
\ref{definition well ordered}. Let $b$ be the smallest element of $B$. Then
$b$ is a lower bound of $B$ and $b\in B$ by Definition \ref{definition bounds}%
. Now $a$ is the smallest element of $M$ by assumption and hence $a\leq b$ by
Definition \ref{definition bounds}. And $b$ is not equal to $a$ since $a\in A$
by assumption and $b\in B$, a set having no element in $A$. And from $a\leq b$
and $a\neq b$ follows $a<b$ by Definition \ref{definition <= and >=}. But then
$b-a<b$ by Theorem \ref{b-a < a}$.$ Thus $b-a$ is less than a lower bound of
$B$ and hence $b-a\notin B$ by Theorem \ref{element less than lower bound}.
Hence $b-a\in A$. But, by assumption, $b-a\in A$ implies $\left(  b-a\right)
+a\in A$. And $\left(  b-a\right)  +a=b$ by Definition
\ref{definition subtraction}. Thus $b$ is an element of $A$ and of $B$ which
is impossible.
\end{proof}

In the following theorem an embedding of a well ordered magnitude space into
an arbitrary magnitude space is constructed inductively. The general approach
is that of Dedekind.\cite{Dedekind}

\begin{theorem}
If $M$ is a well ordered magnitude space with smallest element $a$,
$M^{\prime}$ is an arbitrary magnitude space, and $a^{\prime}\in M^{\prime} $,
then there exists a unique function $\varphi:M\rightarrow M^{\prime}$ such that

(i) $\varphi a=a^{\prime}$ and

(ii) $\varphi b=\varphi\left(  b-a\right)  +a^{\prime}$ for all $b>a$.
\end{theorem}

\begin{proof}
First, there can be at most one function satisfying the two conditions above.
For if there are two distinct functions $\varphi$ and $\psi$ each satisfying
the two conditions, then there must be a smallest $b\in M$ for which $\varphi
b\neq\psi b$ by Definition \ref{definition well ordered}. Now $\varphi
a=a^{\prime}=\psi a$ since each of $\varphi$ and $\psi$ are assumed to have
property (i) above. Thus $a\neq b$ since $\varphi b\neq\psi b$. And $a$ is the
smallest element of $M$ by assumption and hence $a\leq b$ by Definition
\ref{definition bounds}. And from $a\neq b$ and $a\leq b$ follows $a<b$ by
Definition \ref{definition <= and >=}. Hence $b-a<b$ by Theorem \ref{b-a < a}.
And since $b$ is the smallest element of $M$ such that $\varphi b\neq\psi b$,
$\varphi\left(  b-a\right)  =\psi\left(  b-a\right)  $. Therefore%
\begin{align*}
\varphi b  & =\varphi\left(  b-a\right)  +a^{\prime}\text{ (property (ii)
above)}\\
& =\psi\left(  b-a\right)  +a^{\prime}\text{ (Theorem
\ref{equals added to equals or unequals})}\\
& =\psi b\text{ (property (ii) above)}%
\end{align*}
which is a contradiction.

It remains to show that there exists a function $\varphi$ with the specified
properties. To this end, for each $b\in M$ let $M_{b}$ be the set of those
elements in $M$ which are less than or equal to $b$. I say that for each $b\in
M$ there is a unique function $\varphi_{b}:M_{b}\rightarrow M^{\prime}$ such that

(i) $\varphi_{b}a=a^{\prime}$ and

(ii) $\varphi_{b}c=\varphi_{b}\left(  c-a\right)  +a^{\prime}$ for $a<c\leq b$.

That there cannot be two distinct functions with these properties for a given
$b\in M$ can be shown by the same argument as given in the beginning of the
proof. To prove the existence of one such function for each $b\in M$, let $A$
be the set of all elements $b$ in $M$ for which there is a unique function
$\varphi_{b}:M_{b}\rightarrow M^{\prime}$ as described above. In the case of
$b=a$, $M_{a}=\left\{  a\right\}  $ and the function $\varphi_{a}%
:M_{a}\rightarrow M^{\prime}$ defined by $\varphi_{a}a=a^{\prime}$ has the
required properties. Thus $a\in A$. Now suppose $b\in A$. We can then define a
function $\varphi_{b+a}:M_{b+a}\rightarrow M^{\prime}$ in terms of the unique
function $\varphi_{b}:M_{b}\rightarrow M^{\prime}$ according to
\[
\varphi_{b+a}c=\left\{
\begin{array}
[c]{rrr}%
\varphi_{b}c & \text{if} & c\leq b\\
\varphi_{b}b+a^{\prime} & \text{if} & c=b+a
\end{array}
\right.  .
\]
And since $b\in A$, (i) $\varphi_{b+a}a=a^{\prime}$ and (ii) $\varphi
_{b+a}c=\varphi_{b+a}\left(  c-a\right)  +a^{\prime}$ for all $c\in M_{b+a}$.
We have now shown that if $b\in A$, then $b+a\in A$. Hence $A=M$ by the
preceding theorem and for each $b\in M$ there exists a unique function
$\varphi_{b}$ satisfying the two conditions above.

Now define the function $\varphi:M\rightarrow S$ according to $\varphi
b=\varphi_{b}b$. We then have%
\[
\varphi a=\varphi_{a}a=a^{\prime}%
\]
and for any $b>a$%
\[
\varphi b=\varphi_{b}b=\varphi_{b}\left(  b-a\right)  +a^{\prime}%
=\varphi_{b-a}\left(  b-a\right)  +a^{\prime}=\varphi\left(  b-a\right)
+a^{\prime}\text{.}%
\]

\end{proof}

\begin{theorem}
\label{embedding discrete magnitude}If $M$ is a well ordered magnitude space,
$M^{\prime}$ is any magnitude space, $a$ is the smallest element in $M$, and
$a^{\prime}$ is any element in $M^{\prime}$, then there exists a unique
embedding $\varphi:M\rightarrow M^{\prime}$ such that $\varphi a=a^{\prime}$.
\end{theorem}

\begin{proof}
From the preceding theorem there is a unique function $\varphi:M\rightarrow
M^{\prime}$ such that (i) $\varphi a=a^{\prime}$ and (ii) $\varphi
b=\varphi\left(  b-a\right)  +a^{\prime}$ for any $b>a$. Note that from the
second property it follows that
\begin{align*}
\varphi\left(  b+a\right)   & =\varphi\left(  \left(  b+a\right)  -a\right)
+a^{\prime}\text{ (definition of }\varphi\text{)}\\
& =\varphi b+a^{\prime}\text{. (Definition \ref{definition subtraction})}%
\end{align*}
for any $b\in M$ since $b+a>a$.

I say that $\varphi$ is an embedding. To see this, fix $c\in M$ and define $A$
to be the set of all elements $d$ in $M$ such that $\varphi\left(  c+d\right)
=\varphi c+\varphi d$. Now $a\in A$ since $\varphi\left(  c+a\right)  =\varphi
c+a^{\prime}=\varphi c+\varphi a$ by note above. And if $d\in A$, then%
\begin{align*}
\varphi\left(  c+\left(  d+a\right)  \right)   & =\varphi\left(  \left(
c+d\right)  +a\right)  \text{ (associativity)}\\
& =\varphi\left(  c+d\right)  +a^{\prime}\text{ (note above)}\\
& =\varphi c+\varphi d+a^{\prime}\text{ (}d\in A\text{)}\\
& =\varphi c+\varphi\left(  d+a\right)  \text{ (note above)}%
\end{align*}
and so $d+a\in A$. Therefore, $A=M$ by Theorem \ref{induction}. And since our
choice of $c\in M$ was arbitrary, $\varphi\left(  c+d\right)  =\varphi
c+\varphi d$ for all $c,d\in M$. Thus $\varphi$ is an embedding by Definition
\ref{definition homomorphism} and Theorem \ref{morphisms 1-1}.

Now suppose $\psi$ is some other embedding of $M$ into $M^{\prime}$ which maps
$a$ into $a^{\prime}$. Then $\psi a=a^{\prime}$and, if $b>a$ then%
\begin{align*}
\psi b  & =\psi\left(  \left(  b-a\right)  +a\right)  \text{ (Definition
\ref{definition subtraction})}\\
& =\psi\left(  b-a\right)  +\psi a\text{ (Definition
\ref{definition homomorphism})}\\
& =\psi\left(  b-a\right)  +a^{\prime}\text{. (assumed property of }%
\varphi\text{)}%
\end{align*}
But by the preceding theorem, there is only one such mapping and hence
$\psi=\varphi$.
\end{proof}

\begin{theorem}
\label{well ordered magnitude spaces isomorphic}Any two well ordered magnitude
spaces are isomorphic.
\end{theorem}

\begin{proof}
Assume that $M$ and $M^{\prime}$ are well ordered magnitude spaces with
smallest elements $a$ and $a^{\prime}$ respectively. There is an embedding
$\varphi:M\rightarrow M^{\prime}$ such that $\varphi a=a^{\prime}$ and an
embedding $\varphi^{\prime}:M^{\prime}\rightarrow M$ such that $\varphi
^{\prime}a^{\prime}=a$ from the preceding theorem. The composition
$\varphi\varphi^{\prime}:M^{\prime}\rightarrow M^{\prime}$ is an embedding by
Theorem \ref{composition of homomophisms} and the identity mapping
$i_{M^{\prime}}:M^{\prime}\rightarrow M^{\prime}$ is also an embedding by
Theorem \ref{identity homomorphism}. But%
\[
\left(  \varphi\varphi^{\prime}\right)  a^{\prime}=\varphi\left(
\varphi^{\prime}a^{\prime}\right)  =\varphi a=a^{\prime}=i_{M^{\prime}%
}a^{\prime}%
\]
and so $\varphi\varphi^{\prime}=i_{M^{\prime}}$ by the preceding theorem.
Therefore $\varphi$ is onto by Theorem \ref{proving function is onto} and
hence is an isomorphism according to Definition
\ref{definition isomorphism isomorphic}.
\end{proof}

\section{Natural Numbers and Integral Multiples\label{integral multiples}}

\begin{definition}
\label{definition natural numbers}Pick any well ordered magnitude space and
denote it by $%
\mathbb{N}
$ and denote the smallest element of $%
\mathbb{N}
$ by $1$. Since all well ordered magnitude spaces are isomorphic (Theorem
\ref{well ordered magnitude spaces isomorphic}) and we will use only the
algebraic properties that the well ordered magnitude spaces have in common, it
does not matter which well ordered magnitude space is chosen to play the role
of the \textquotedblleft number system\textquotedblright\ $%
\mathbb{N}
$. We call $%
\mathbb{N}
$ the \textbf{natural numbers}.
\end{definition}

\begin{definition}
\label{definition integral multiple}If $n\in%
\mathbb{N}
$ and $a$ is an element of a magnitude space $M$, then the \textbf{integral
multiple} $na$ is given by $na=\varphi_{1,a}n$ where $\varphi_{1,a}$ is the
unique embedding of $%
\mathbb{N}
$ into $M$ which maps $1$ into $a$ (Theorem \ref{embedding discrete magnitude}).
\end{definition}

\begin{theorem}
\label{1a and (n+1)a}For each $a\in M$, $\left(  n+1\right)  a=na+a$.
\end{theorem}

\begin{proof}
Let $\varphi_{1,a}$ be the unique embedding of $%
\mathbb{N}
$ into $M$ which maps $1$ into $a$. Then
\begin{align*}
\left(  n+1\right)  a  & =\varphi_{1,a}\left(  n+1\right)  \text{ (Definition
\ref{definition integral multiple})}\\
& =\varphi_{1,a}n+\varphi_{1,a}1\text{ (Definition
\ref{definition homomorphism})}\\
& =na+a\text{. (Definition \ref{definition integral multiple})}%
\end{align*}

\end{proof}

\begin{remark}
For each $a$ in a magnitude space $M$%
\begin{align*}
1a  & =\varphi_{1,a}1=a\\
\left(  1+1\right)  a  & =\varphi_{1,a}\left(  1+1\right)  =\varphi
_{1,a}1+a=a+a\\
\left(  1+1+1\right)  a  & =\varphi_{1,a}\left(  1+1+1\right)  =\varphi
_{1,a}\left(  1+1\right)  +a=a+a+a\\
& \vdots
\end{align*}

\end{remark}

\begin{theorem}
\label{multiples of morphisms}If $\chi:M\rightarrow M^{\prime}$ is an
embedding, then $\chi\left(  na\right)  =n\left(  \chi a\right)  $ for every
$a\in M$ and $n\in%
\mathbb{N}
$.
\end{theorem}

\begin{proof}
Let $a$ be fixed. There exists a unique embedding $\varphi_{1,a}$ of $%
\mathbb{N}
$ into $M$ which maps $1$ into $a$ and a unique embedding $\varphi_{1,\chi a}$
of $%
\mathbb{N}
$ into $M^{\prime}$ which maps $1$ into $\chi a$ by Theorem
\ref{embedding discrete magnitude}. Now $\chi\varphi_{1,a}:%
\mathbb{N}
\rightarrow M^{\prime}$ is an embedding by Theorem
\ref{composition of homomophisms} and%
\begin{align*}
\left(  \chi\varphi_{1,a}\right)  1  & =\chi\left(  \varphi_{1,a}1\right)
\text{ (Definition \ref{definition composition})}\\
& =\chi\left(  a\right)  .\text{ (definition of }\varphi_{1,a}\text{)}%
\end{align*}
Therefore each of $\varphi_{1,\chi a}$ and $\chi\varphi_{1,a}$ are embeddings
of $%
\mathbb{N}
$ into $M^{\prime}$ which map $1$ (the smallest element in $%
\mathbb{N}
$) into the same element $\chi a\in M^{\prime}$ and hence $\varphi_{1,\chi
a}=\chi\varphi_{1,a}$ by Theorem \ref{embedding discrete magnitude}.
Therefore
\begin{align*}
n\left(  \chi a\right)   & =\varphi_{1,\chi a}n\text{ (Definition
\ref{definition integral multiple})}\\
& =\left(  \chi\varphi_{1,a}\right)  n\text{ (Definition
\ref{definition equal functions})}\\
& =\chi\left(  \varphi_{1,a}n\right)  \text{ (Definition
\ref{definition composition})}\\
& =\chi\left(  na\right)  \text{. (Definition
\ref{definition integral multiple})}%
\end{align*}

\end{proof}

\begin{theorem}
\label{prelude to homomorphisms are proportional}If $\varphi:M\rightarrow
M^{\prime}$ is an embedding and $a,b\in M$, then for each pair $m,n\in%
\mathbb{N}
$, $m\left(  \varphi a\right)  $ has to $n\left(  \varphi b\right)  $ the same
relation (%
$<$%
, =, or
$>$%
) as $ma$ has to $nb$.
\end{theorem}

\begin{proof}
First, $m\left(  \varphi a\right)  =\varphi\left(  ma\right)  $ and $n\left(
\varphi b\right)  =\varphi\left(  nb\right)  $ by the preceding theorem.
Second, $\varphi\left(  ma\right)  $ has to $\varphi\left(  nb\right)  $ the
same relation (%
$<$%
, =, or
$>$%
) as $ma$ has to $nb$ by Theorem \ref{morphisms 1-1}. Therefore $m\left(
\varphi a\right)  $ has to $n\left(  \varphi b\right)  $ the same relation (%
$<$%
, =, or
$>$%
) as $ma$ has to $nb$.
\end{proof}

\section{Embeddings and Ratios}

The preceding theorem is the connecting point between modern algebraic
definitions of number systems and the classical theory of ratios. The
classical theory is based on the following two definitions.

\begin{definition}
[Euclid V, Definition 4]\label{definition have a ratio}Two elements $a,b$ in a
magnitude space are said to \textbf{have a ratio} if there are natural numbers
$m,n$ such that $ma>b$ and $nb>a$.
\end{definition}

\begin{definition}
[Euclid V, Definitions 5, 6, and 7]\label{revised euclid def of ratios}Let $M$
and $M^{\prime}$ be magnitude spaces. We say that a pair $a,b\in M$ has the
\textbf{same ratio} as (or are \textbf{proportional} to) a pair $a^{\prime
},b^{\prime}\in M^{\prime}$ if $ma$ has to $nb$ the same relation (%
$<$%
, =, or
$>$%
) as $ma^{\prime}$ has to $nb^{\prime}$ for every $m,n\in%
\mathbb{N}
$. And in this case we write $a:b=a^{\prime}:b^{\prime}$. And if for some
$m,n\in%
\mathbb{N}
$, $ma>nb$ and $ma^{\prime}\leq nb^{\prime}$, then we say $a$ has to $b$ a
\textbf{greater ratio} than $a^{\prime}$ has to $b^{\prime}$ and we write
$a:b>a^{\prime}:b^{\prime}$ or $a^{\prime}:b^{\prime}<a:b$.
\end{definition}

We are particularly interested in magnitude spaces in which every pair of
elements have a ratio. The next two theorems establish that a magnitude space
has this property if and only if it is an Archimedean magnitude space.

\begin{theorem}
If every pair of elements of a magnitude space have a ratio, then the
magnitude space is Archimedean.
\end{theorem}

\begin{proof}
Assume $M$ is a magnitude space in which every pair of elements have a ratio.
Let $a\in M$ and let $A$ be a nonempty subset of $M$ which has an upper bound.
I say there is some $\zeta\in M$ such that $\zeta\in A$ and $\zeta+a\notin A$.

Let $B$ be the set of all natural numbers $n$ such that $na$ is an upper bound
of $A$. If $b$ is any upper bound of $A$, then there is some natural number
$n$ such that $na>b$ because $a$ and $b$ have a ratio by assumption and
Definition \ref{definition have a ratio}$.$ Thus $B$ is nonempty. And the
natural numbers are a well ordered magnitude space by Definition
\ref{definition natural numbers}. Hence $B$ has a smallest element by
Definition \ref{definition well ordered}$.$ Let $n$ be the smallest element of
$B$.

Case 1: $n=1$. Pick any $\zeta\in A$. Then $na<\zeta+a$.

Case 2: $n>1$. Since $n$ is the smallest natural number such that $na$ is an
upper bound of $A$ and $n-1<n$ by Theorem \ref{b-a < a}, it follows that
$\left(  n-1\right)  a$ is not an upper bound of $A$. Therefore there must be
some $\zeta\in A$ such that $\left(  n-1\right)  a<\zeta$ by Definition
\ref{definition bounds}. Thus $\left(  n-1\right)  a+a<\zeta+a$ by Theorem
\ref{Translation by a} and
\begin{align*}
\left(  n-1\right)  a+a  & =\left(  \left(  n-1\right)  +1\right)  a\text{
(Theorem \ref{1a and (n+1)a})}\\
& =na\text{. (Definition \ref{definition subtraction})}%
\end{align*}
Therefore $na<\zeta+a$.

In both cases, $\zeta\in A$ and $\zeta+a$ is greater than an upper bound of
$A$ and hence $\zeta+a\notin A$ by Theorem
\ref{element greater than upper bound}. Therefore $M$ is an Archimedean
magnitude space according to Definition \ref{definition Archimedean}.
\end{proof}

\begin{theorem}
\label{some n such that na>b}If $a$ and $b$ are elements of an Archimedean
magnitude space, then there is some natural number $n$ such that $na>b$.
\end{theorem}

\begin{proof}
Suppose $na\leq b$ for all natural numbers $n$. Let $A$ be the set of all
integral multiples of $a$. Then $b$ is an upper bound of $A$ by Definition
\ref{definition bounds}. Thus there is some element $\zeta\in M$ such that
$\zeta\in A$ and $\zeta+a\notin A$ by Definition \ref{definition Archimedean}.
But if $\zeta\in A$, then $\zeta=na$ for some natural number $n$ and
$\zeta+a=na+a$. And $na+a=\left(  n+1\right)  a$ by Theorem
\ref{1a and (n+1)a}. Hence $\zeta+a$ is an integral multiple of $a$ and hence
$\zeta+a\in A$ which is a contradiction.
\end{proof}

\begin{theorem}
Any two elements of an Archimedean magnitude space have a ratio.
\end{theorem}

\begin{proof}
Previous theorem and Definition \ref{definition have a ratio}.
\end{proof}

\begin{theorem}
\label{homomorphisms are proportional}If $M$ and $M^{\prime}$ are Archimedean
magnitude spaces and $\varphi:M\rightarrow M^{\prime}$ is an embedding, then
$\varphi a:\varphi b=a:b$ for every $a,b\in M.$
\end{theorem}

\begin{proof}
Any two elements of $M$ have a ratio from the previous theorem. And likewise
any two elements of $M^{\prime}$ have a ratio. Fix $a,b\in M$. For any two
natural numbers $m$ and $n$, $m\left(  \varphi a\right)  $ has to $n\left(
\varphi b\right)  $ the same relation (%
$<$%
, =, or
$>$%
) as $ma$ has to $nb$ by Theorem
\ref{prelude to homomorphisms are proportional}. Therefore $\varphi a:\varphi
b=a:b$ according to Definition \ref{revised euclid def of ratios}.
\end{proof}

\section{Classical Theory of Ratios}

Henceforth we shall consider only Archimedean magnitude spaces.

In this section variables $a,b,c,d,e,f$ are all elements of a magnitude space
$M$, $a^{\prime},b^{\prime},c^{\prime},d^{\prime},e^{\prime},f^{\prime}$ are
all elements of a magnitude space $M^{\prime}$, and so on. Variables $j,k,m,n$
are natural numbers.

The propositions in this section appear in the exact order as the propositions
in Book V of The Elements. The proofs are likewise similar with the following exceptions.

\begin{enumerate}
\item Proofs for Propositions \ref{V 1}, \ref{V 2}, and \ref{V 3} are based on
Definition \ref{definition integral multiple} and Theorem
\ref{embedding discrete magnitude}.

\item The proofs of Propositions \ref{V 10} and \ref{V 18} follow those of
Robert Simson.\cite{Simson}
\end{enumerate}

\begin{proposition}
\label{V 1}$n(a+b)=na+nb$.
\end{proposition}

\begin{proof}
Let $a$ and $b$ be fixed. There exists a unique embeddings $\varphi_{1,a}$,
$\varphi_{1,b}$, and $\varphi_{1,a+b}$ of $%
\mathbb{N}
$ into $M$ which map $1$ into $a$, $b$, and $a+b$ respectively by Theorem
\ref{embedding discrete magnitude}. Note that $\varphi_{1,a}+\varphi_{1,b}$ is
an embedding of $%
\mathbb{N}
$ into $M$ by Theorem \ref{sum of homomorphisms} and%
\begin{align*}
\left(  \varphi_{1,a}+\varphi_{1,b}\right)  1  & =\varphi_{1,a}1+\varphi
_{1,b}1\text{ (Definition \ref{definition sum of embeddings})}\\
& =a+b\text{. (definitions of }\varphi_{1,a}\text{ and }\varphi_{1,b}\text{)}%
\end{align*}
Thus $\varphi_{1,a+b}$ and $\varphi_{1,a}+\varphi_{1,b}$ are each embeddings
of $%
\mathbb{N}
$ into $M$ which map $1$ into $a+b$ and hence $\varphi_{1,a+b}=\varphi
_{1,a}+\varphi_{1,b}$ by Theorem \ref{embedding discrete magnitude}. Therefore%
\begin{align*}
n(a+b)  & =\varphi_{1,a+b}n\text{ (Definition
\ref{definition integral multiple})}\\
& =\left(  \varphi_{1,a}+\varphi_{1,b}\right)  n\text{ (Definition
\ref{definition equal functions})}\\
& =\varphi_{1,a}n+\varphi_{1,b}n\text{ (Definition
\ref{definition sum of embeddings})}\\
& =na+nb\text{. (Definition \ref{definition integral multiple})}%
\end{align*}

\end{proof}

\begin{proposition}
\label{V 2}$\left(  m+n\right)  a=ma+na$.
\end{proposition}

\begin{proof}
Let $a$ be fixed. There exists a unique embedding $\varphi_{1,a}$ of $%
\mathbb{N}
$ into $M$ which maps $1$ into $a$ by Theorem
\ref{embedding discrete magnitude} and%
\begin{align*}
\left(  m+n\right)  a  & =\varphi_{1,a}\left(  m+n\right)  \text{ (Definition
\ref{definition integral multiple})}\\
& =\varphi_{1,a}m+\varphi_{1,a}n\text{ (Definition
\ref{definition homomorphism})}\\
& =ma+na\text{. (Definition \ref{definition integral multiple})}%
\end{align*}

\end{proof}

\begin{proposition}
\label{V 3}$\left(  mn\right)  a=m\left(  na\right)  $.
\end{proposition}

\begin{proof}
Let $a$ be fixed. There exists a unique embedding $\varphi_{1,a}$ of $%
\mathbb{N}
$ into $M$ which maps $1$ into $a$ by Theorem
\ref{embedding discrete magnitude} and%
\begin{align*}
\left(  mn\right)  a  & =\varphi_{1,a}\left(  mn\right)  \text{ (Definition
\ref{definition integral multiple})}\\
& =m\left(  \varphi_{1,a}n\right)  \text{ (Theorem
\ref{multiples of morphisms})}\\
& =m\left(  na\right)  \text{. (Definition \ref{definition integral multiple}%
)}%
\end{align*}

\end{proof}

\begin{proposition}
\label{V 4}If $a:b=a^{\prime}:b^{\prime}$, then $ja:kb=ja^{\prime}:kb^{\prime
}$.
\end{proposition}

\begin{proof}
If $a:b=a^{\prime}:b^{\prime}$, then for any two numbers $m$ and $n$%
\begin{align*}
m\left(  ja\right)  >n\left(  kb\right)   & \implies\left(  mj\right)
a>\left(  nk\right)  b\text{ (Proposition \ref{V 3})}\\
& \implies\left(  mj\right)  a^{\prime}>\left(  nk\right)  b^{\prime}\text{
(Definition \ref{revised euclid def of ratios})}\\
& \implies m\left(  ja^{\prime}\right)  >n\left(  kb^{\prime}\right)  \text{.
(Proposition \ref{V 3}).}%
\end{align*}
And the same argument applies with
$>$
replaced by = or
$<$%
$.$ Therefore $ja:kb=ja^{\prime}:kb^{\prime}$ according to Definition
\ref{revised euclid def of ratios}.
\end{proof}

\begin{proposition}
\label{V 5}If $a>b$, then $na>nb$ and $na-nb=n\left(  a-b\right)  $. And, more
generally, $na$ has to $nb$ the same relation (%
$<$%
, =, or
$>$%
) as $a$ has to $b$.
\end{proposition}

\begin{proof}
Fix $n$ and let $\varphi:M\rightarrow M$ be the function defined by $\varphi
a=na$. Then $\varphi\left(  a+b\right)  =\varphi a+$ $\varphi b$ for any $a$
and $b$ by Proposition \ref{V 1} and hence $\varphi$ is a homomorphism
according to Definition \ref{definition homomorphism}. Hence Theorems
\ref{morphisms monotonic} and \ref{morphisms 1-1} apply.
\end{proof}

\begin{proposition}
\label{V 6}If $m>n$, then $ma>na$ and $ma-na=(m-n)a$. And, more generally,
$ma$ has to $na$ the same relation (%
$<$%
, =, or
$>$%
) as $m$ has to $n$.
\end{proposition}

\begin{proof}
Fix $a$ and let $\varphi:%
\mathbb{N}
\rightarrow M$ be the function defined by $\varphi n=na$. Then $\varphi\left(
m+n\right)  =\varphi m+$ $\varphi n$ for any $m$ and $n$ by Proposition
\ref{V 2} and hence $\varphi$ is a homomorphism according to Definition
\ref{definition homomorphism}. Hence Theorems \ref{morphisms monotonic} and
\ref{morphisms 1-1} apply.
\end{proof}

\begin{proposition}
\label{V 7}If $a=b$, then $a:c=b:c$ and $c:a=c:b.$
\end{proposition}

\begin{proof}
If $a=b$, then $ma=mb$ by Proposition \ref{V 5}. Thus for any $m$ and $n$,
$ma$ has the same relation to $nc$ (%
$<$%
, =, or
$>$%
) as $mb$ has to $nc$. Therefore $a:c=b:c$ according to Definition
\ref{revised euclid def of ratios}. A similar argument shows that $c:a=c:b$.
\end{proof}

\begin{proposition}
\label{V 8}If $a>b$, then $a:c>b:c$ and $c:b>c:a$.
\end{proposition}

\begin{proof}
Assume $a>b$. Since magnitude spaces are assumed to be Archimedean, there is
some $m$ such that $m\left(  a-b\right)  >c$ by Theorem
\ref{some n such that na>b}. And from Proposition \ref{V 5}, $ma>mb$ and
$m\left(  a-b\right)  =ma-mb $. Therefore $ma-mb>c$. And $ma>mb+c$ by Theorem
\ref{moving b accross <=>}.

Likewise, the set $A$ of all natural numbers $k$ such that $kc>mb$ is nonempty
by Theorem \ref{some n such that na>b}. And since the natural numbers are well
ordered by Definition \ref{definition natural numbers}, $A$ has a smallest
element by Definition \ref{definition well ordered}. Let $n$ be the smallest
element of $A$.

I say $ma>nc$. Suppose, on the contrary, that $nc\geq ma$. Then, since it was
shown that $ma>mb+c$, it follows that $nc>mb+c$ by Theorem
\ref{Transitivity mixed} and%
\begin{align*}
nc>mb+c  & \implies nc>mb+1c\text{ (Definition
\ref{definition integral multiple})}\\
& \implies nc-1c>mb\text{ (Theorem \ref{moving b accross <=>})}\\
& \implies\left(  n-1\right)  c>mb\text{ (Proposition \ref{V 5})}\\
& \implies\left(  n-1\right)  \in A\text{. (definition of }A\text{)}%
\end{align*}
But $\left(  n-1\right)  <n$ by Theorem \ref{b-a < a} and since $n$ is the
smallest element of $A$, $n$ is a lower bound of $A$ by Definition
\ref{definition bounds}, and hence $\left(  n-1\right)  \notin A$ by Theorem
\ref{element less than lower bound} which is a contradiction. Therefore,
indeed, $ma>nc$.

And $n$ was chosen so that $nc>mb$. Therefore, $ma>nc$ and $mb\leq nc$ by
Definition \ref{definition <= and >=} and hence $a:c>b:c$ according to
Definition \ref{revised euclid def of ratios}.

Likewise, $nc>mb$ and $nc\leq ma$ by Definition \ref{definition <= and >=} and
hence $c:b>c:a$\ according to Definition \ref{revised euclid def of ratios}.
\end{proof}

\begin{proposition}
\label{V 9}If $a:c=b:c$, then $a=b$. And if $c:a=c:b$ then $a=b$.
\end{proposition}

\begin{proof}
Assume $a:c=b:c$. If $a>b$, then $a:c>b:c$ by the preceding theorem. But this
is not consistent with the assumption. And for the same reason $b>a$ is not
consistent with the assumption. Therefore $a=b$. The second part of the
theorem is proved in a similar manner.
\end{proof}

\begin{proposition}
\label{V 10}If $a:c>b:c$, then $a>b$. And if $c:a>c:b$, then $b>a$.
\end{proposition}

\begin{proof}
Assume $a:c>b:c$. Then there are numbers $m$ and $n$ such that $ma>nc$ and
$mb\leq nc$ by Definition \ref{revised euclid def of ratios}. And from $ma>nc$
and $nc\geq mb$ follows $ma>mb$ by Theorem \ref{Transitivity mixed}.
Therefore, $a>b$ by Proposition \ref{V 5}. The second part of the theorem is
proved in the same way.
\end{proof}

\begin{proposition}
\label{V 11}If $a:b=a^{\prime}:b^{\prime}$ and $a^{\prime\prime}%
:b^{\prime\prime}=a^{\prime}:b^{\prime}$, then $a:b=a^{\prime\prime}%
:b^{\prime\prime}$.
\end{proposition}

\begin{proof}
If $a:b=a^{\prime}:b^{\prime}$ and $a^{\prime\prime}:b^{\prime\prime
}=a^{\prime}:b^{\prime}$, and $m$ and $n$ are any two numbers%
\begin{align*}
ma>nb  & \implies ma^{\prime}>nb^{\prime}\text{ (Definition
\ref{revised euclid def of ratios})}\\
& \implies ma^{\prime\prime}>nb^{\prime\prime}\text{. (Definition
\ref{revised euclid def of ratios})}%
\end{align*}
And the same argument applies with
$>$
replaced by = or
$<$%
$.$ Therefore $a:b=a^{\prime\prime}:b^{\prime\prime}$ by Definition
\ref{revised euclid def of ratios}.
\end{proof}

\begin{proposition}
\label{V 12}If $a:b=c:d$ then $a:b=\left(  a+c\right)  :\left(  b+d\right)  $.
\end{proposition}

\begin{proof}
Assume $a:b=c:d$. If $ma>nb$, then $mc>nd$ by Definition
\ref{revised euclid def of ratios} and hence%
\begin{align*}
m\left(  a+c\right)   & =ma+mc\text{ (Proposition \ref{V 1})}\\
& >nb+mc\text{ (Theorem \ref{equals added to equals or unequals})}\\
& >nb+nd\text{ (Theorem \ref{equals added to equals or unequals})}\\
& =n\left(  b+d\right)  \text{. (Proposition \ref{V 1})}%
\end{align*}
Therefore $ma>nb\implies m\left(  a+c\right)  >n\left(  b+d\right)  $. And the
same argument applies with
$>$
replaced by = or
$<$%
$.$ Therefore $a:b=\left(  a+c\right)  :\left(  b+d\right)  $ by Definition
\ref{revised euclid def of ratios}.
\end{proof}

\begin{proposition}
\label{V 13}If $a:b=a^{\prime}:b^{\prime}$ and $a^{\prime}:b^{\prime
}>a^{\prime\prime}:b^{\prime\prime}$, then $a:b>a^{\prime\prime}%
:b^{\prime\prime}$. And if $a:b=a^{\prime}:b^{\prime}$ and $a^{\prime\prime
}:b^{\prime\prime}>a^{\prime}:b^{\prime}$, then $a^{\prime\prime}%
:b^{\prime\prime}>a:b$.
\end{proposition}

\begin{proof}
Assume $a:b=a^{\prime}:b^{\prime}$ and $a^{\prime}:b^{\prime}>a^{\prime\prime
}:b^{\prime\prime}$. Then there are numbers $m$ and $n$ such that $ma^{\prime
}>nb^{\prime}$ and $ma^{\prime\prime}\leq nb^{\prime\prime}$ by Definition
\ref{revised euclid def of ratios}. And because $ma^{\prime}>nb^{\prime}$,
also $ma>nb$ by Definition \ref{revised euclid def of ratios}. Therefore
$ma>nb$ and $ma^{\prime\prime}\leq nb^{\prime\prime}$ and hence $a:b>a^{\prime
\prime}:b^{\prime\prime}$ by Definition \ref{revised euclid def of ratios}.
The second part of the theorem is proved in a similar manner.
\end{proof}

\begin{proposition}
\label{V 14}If $a:b=c:d$, then $a$ has the same relation to $c$ (%
$<$%
, =, or
$>$%
) as $b$ has to $d$.
\end{proposition}

\begin{proof}
Assume $a:b=c:d$. Then%
\begin{align*}
a>c  & \implies a:b>c:b\text{ (Proposition \ref{V 8})}\\
& \implies c:d>c:b\text{ (Proposition \ref{V 13})}\\
& \implies b>d\text{ (Proposition \ref{V 10})}%
\end{align*}
and by the same argument $a<c\implies b<d$. Likewise%
\begin{align*}
a=c  & \implies a:b=c:b\text{ (Proposition \ref{V 7})}\\
& \implies c:d=c:b\text{ (Proposition \ref{V 11})}\\
& \implies b=d\text{. (Proposition \ref{V 9})}%
\end{align*}

\end{proof}

\begin{proposition}
\label{V 15}$a:b=ka:kb$.
\end{proposition}

\begin{proof}
The theorem is true for $k=1$ since $1a=a$ and $1b=b$ by Definition
\ref{definition integral multiple}. Now assume the theorem is true for $k$.
Then $a:b=ka:kb$ and hence $\left(  ka+a\right)  :\left(  kb+b\right)  =a:b$
by Proposition \ref{V 12}. But $ka+a=\left(  k+1\right)  a$ and $kb+b=\left(
k+1\right)  b$ Theorem \ref{1a and (n+1)a}. Therefore$\left(  k+1\right)
a:\left(  k+1\right)  b=a:b$ and the theorem is true for $k+1$. Therefore the
theorem is true for all $k$ by Theorem \ref{induction}.
\end{proof}

\begin{proposition}
\label{V 16}If $a:b=c:d$, then $a:c=b:d$.
\end{proposition}

\begin{proof}
Assume $a:b=c:d$ and let $m,n$ be any two natural numbers. Then%
\[
ma:mb=a:b\text{ and }nc:nd=c:d\text{.}%
\]
by the preceding theorem and hence%
\[
ma:mb=nc:nd
\]
by Proposition \ref{V 11}. Therefore $ma$ has the same relation to $nc$ (%
$<$%
, =, or
$>$%
) as $mb$ has to $nd$ by Proposition \ref{V 14} and hence $a:c=b:d$ by
Definition \ref{revised euclid def of ratios}.
\end{proof}

\begin{remark}
An alternate formulation of the following theorem is that $a:b=a^{\prime
}:b^{\prime}$ implies $\left(  a-b\right)  :b=\left(  a^{\prime}-b^{\prime
}\right)  :b^{\prime}$
\end{remark}

\begin{proposition}
\label{V 17}If $\left(  a+b\right)  :b=\left(  a^{\prime}+b^{\prime}\right)
:b^{\prime}$, then $a:b=a^{\prime}:b^{\prime}$.
\end{proposition}

\begin{proof}
Assume $\left(  a+b\right)  :b=\left(  a^{\prime}+b^{\prime}\right)
:b^{\prime}$. For any $m$ and $n$%
\begin{align*}
ma>nb  & \implies ma+mb>nb+mb\text{ (Theorem
\ref{equals added to equals or unequals})}\\
& \implies m\left(  a+b\right)  >\left(  n+m\right)  b\text{ (Propositions
\ref{V 1} and \ref{V 2})}\\
& \implies m\left(  a^{\prime}+b^{\prime}\right)  >\left(  n+m\right)
b^{\prime}\text{ (Definition \ref{revised euclid def of ratios})}\\
& \implies ma^{\prime}+mb^{\prime}>nb^{\prime}+mb^{\prime}\text{ (Propositions
\ref{V 1} and \ref{V 2})}\\
& \implies ma^{\prime}>nb^{\prime}\text{. (Theorem
\ref{equals added to equals or unequals})}%
\end{align*}
And the same argument applies with
$>$
replaced by = or
$<$%
$.$ Therefore $a:b=a^{\prime}:b^{\prime}$ according to Definition
\ref{revised euclid def of ratios}.
\end{proof}

\begin{proposition}
\label{V 18}If $a:b=a^{\prime}:b^{\prime}$, then $\left(  a+b\right)
:b=\left(  a^{\prime}+b^{\prime}\right)  :b^{\prime}$.
\end{proposition}

\begin{proof}
Assume $a:b=a^{\prime}:b^{\prime}$ and consider the integral multiples
$m\left(  a+b\right)  $ and $nb$.

Now $a+b>b$ by Definition \ref{definition <} and hence $m\left(  a+b\right)
>mb$ by Proposition \ref{V 5}.

If $m=n$ then $m\left(  a+b\right)  >nb$. And if $m>n$, then $mb>nb$ by
Proposition \ref{V 6} and from $m\left(  a+b\right)  >mb$ and $mb>nb$ follows
$m\left(  a+b\right)  >nb$ by Theorem \ref{Transitivity of >}. In summary, if
$m\geq n$, then $m\left(  a+b\right)  >nb$ and also by the same argument
$m\left(  a^{\prime}+b^{\prime}\right)  >nb^{\prime}$.

Now suppose $m<n$. In this case, $mb<nb$ by Proposition \ref{V 6}. Thus%
\begin{align*}
m\left(  a+b\right)  <nb  & \Longrightarrow ma+mb<nb\text{ (Proposition
\ref{V 1})}\\
& \Longrightarrow ma<nb-mb\text{ (Theorem \ref{moving b accross <=>})}\\
& \Longrightarrow ma<\left(  n-m\right)  b\text{ (Proposition \ref{V 6})}\\
& \Longrightarrow ma^{\prime}<\left(  n-m\right)  b^{\prime}\text{ (Definition
\ref{revised euclid def of ratios})}\\
& \Longrightarrow ma^{\prime}<nb^{\prime}-mb^{\prime}\text{ (Proposition
\ref{V 6})}\\
& \Longrightarrow ma^{\prime}+mb^{\prime}<nb^{\prime}\text{ (Theorem
\ref{moving b accross <=>})}\\
& \Longrightarrow m\left(  a^{\prime}+b^{\prime}\right)  <nb^{\prime}\text{.
(Proposition \ref{V 1})}%
\end{align*}
And the same argument as above applies with
$<$
replaced with = or
$>$%
.

We have now shown that for all $m$ and $n$, $m\left(  a+b\right)  $ has to
$nb$ the same relation (%
$<$%
, =, or
$>$%
) as $m\left(  a^{\prime}+b^{\prime}\right)  $ has to $nb^{\prime}$ and hence
$\left(  a+b\right)  :b=\left(  a^{\prime}+b^{\prime}\right)  :b^{\prime}$
according to Definition \ref{revised euclid def of ratios}.
\end{proof}

\begin{proposition}
\label{V 19}If $\left(  a+b\right)  :\left(  c+d\right)  =a:c$, then $b:d=a:c$.
\end{proposition}

\begin{proof}%
\begin{align*}
\left(  a+b\right)  :\left(  c+d\right)  =a:c  & \Longrightarrow\left(
a+b\right)  :a=\left(  c+d\right)  :c\text{ (Proposition \ref{V 16})}\\
& \Longrightarrow b:a=d:c\text{ (Proposition \ref{V 17})}\\
& \Longrightarrow b:d=a:c\text{. (Proposition \ref{V 16})}%
\end{align*}

\end{proof}

\begin{proposition}
\label{V 20}If $a:b=a^{\prime}:b^{\prime}$ and $b:c=b^{\prime}:c^{\prime} $,
then $a$ has to $c$ the same relation (%
$<$%
, =, or
$>$%
) as $a^{\prime}$ has to $c^{\prime}$.
\end{proposition}

\begin{proof}
Assume $a:b=a^{\prime}:b^{\prime}$ and $b:c=b^{\prime}:c^{\prime}$. Then
$c:b=c^{\prime}:b^{\prime}$ by Definition \ref{revised euclid def of ratios}
and hence%
\begin{align*}
a>c  & \implies a:b>c:b\text{ (Proposition \ref{V 8})}\\
& \implies a^{\prime}:b^{\prime}>c:b\text{ (Proposition \ref{V 13})}\\
& \implies a^{\prime}:b^{\prime}>c^{\prime}:b^{\prime}\text{ (Proposition
\ref{V 13})}\\
& \implies a^{\prime}>c^{\prime}\text{. (Proposition \ref{V 10})}%
\end{align*}
And similar arguments apply with
$>$
replaced by = or
$<$%
$.$
\end{proof}

\begin{proposition}
\label{V 21}If $a:b=b^{\prime}:c^{\prime}$ and $b:c=a^{\prime}:b^{\prime} $,
then $a$ has to $c$ the same relation (%
$<$%
, =, or
$>$%
) as $a^{\prime}$ has to $c^{\prime}$.
\end{proposition}

\begin{proof}
Assume $a:b=b^{\prime}:c^{\prime}$ and $b:c=a^{\prime}:b^{\prime}$. Then
$c:b=b^{\prime}:a^{\prime}$ by Definition \ref{revised euclid def of ratios}
and hence%
\begin{align*}
a>c  & \implies a:b>c:b\text{ (Proposition \ref{V 8})}\\
& \implies b^{\prime}:c^{\prime}>c:b\text{ (Proposition \ref{V 13})}\\
& \implies b^{\prime}:c^{\prime}>b^{\prime}:a^{\prime}\text{ (Proposition
\ref{V 13})}\\
& \implies a^{\prime}>c^{\prime}\text{. (Proposition \ref{V 10})}%
\end{align*}
And similar arguments apply with
$>$
replaced by = or
$<$%
$.$
\end{proof}

\begin{proposition}
\label{V 22}If $a:b=a^{\prime}:b^{\prime}$ and $b:c=b^{\prime}:c^{\prime} $,
then $a:c=a^{\prime}:c^{\prime}$.
\end{proposition}

\begin{proof}
Assume $a:b=a^{\prime}:b^{\prime}$ and $b:c=b^{\prime}:c^{\prime}$ and let $m$
and $n$ be any two natural numbers. Then%
\[
\left(  ma\right)  :\left(  mb\right)  =\left(  ma^{\prime}\right)  :\left(
mb^{\prime}\right)  \text{ and }\left(  mb\right)  :\left(  nc\right)
=\left(  mb^{\prime}\right)  :\left(  nc^{\prime}\right)
\]
by Proposition \ref{V 4}. And $ma$ has to $nc$ the same relation (%
$<$%
, =, or
$>$%
) as $ma^{\prime}$ has to $nc^{\prime}$ by Proposition \ref{V 20}. Therefore
$a:c=a^{\prime}:c^{\prime}$ according to Definition
\ref{revised euclid def of ratios}.
\end{proof}

\begin{proposition}
\label{V 23}If $a:b=b^{\prime}:c^{\prime}$ and $b:c=a^{\prime}:b^{\prime} $,
then $a:c=a^{\prime}:c^{\prime}$.
\end{proposition}

\begin{proof}
Assume $a:b=b^{\prime}:c^{\prime}$ and $b:c=a^{\prime}:b^{\prime}$ and let $m$
and $n$ be any two natural numbers. Then%
\[
\left(  ma\right)  :\left(  nb\right)  =\left(  mb^{\prime}\right)  :\left(
nc^{\prime}\right)  \text{ and }\left(  mb\right)  :\left(  nc\right)
=\left(  ma^{\prime}\right)  :\left(  nb^{\prime}\right)  \text{.}%
\]
by Proposition \ref{V 4}. And $ma$ has to $nc$ the same relation (%
$<$%
, =, or
$>$%
)as $ma^{\prime}$ has to $nc^{\prime}$ by Proposition \ref{V 21}. Therefore
$a:c=a^{\prime}:c^{\prime}$ according to Definition
\ref{revised euclid def of ratios}.
\end{proof}

\begin{proposition}
\label{V 24}If $a:b=c:d$ and $e:b=f:d$, then $\left(  a+e\right)  :b=\left(
c+f\right)  :d$.
\end{proposition}

\begin{proof}
Assume $a:b=c:d$ and $e:b=f:d$ or, equivalently,%
\[
a:b=c:d\text{ and }b:e=d:f\text{.}%
\]
Then $a:e=c:f$ by Proposition \ref{V 22} and $\left(  a+e\right)  :e=\left(
c+f\right)  :f$ by Proposition \ref{V 18}. And $e:b=f:d$ by assumption.
Therefore $\left(  a+e\right)  :b=\left(  c+f\right)  :d$ by Proposition
\ref{V 22}.
\end{proof}

\section{Embeddings and the Fourth Proportional}

In this section $a,b,c$ are elements of a magnitude space $M$ and $a^{\prime
},b^{\prime},c^{\prime}$ are elements of a magnitude space $M^{\prime}$.

\begin{definition}
\label{definition fourth proportional}If $a:b=a^{\prime}:b^{\prime}$, then we
say that $b^{\prime}$ is a \textbf{fourth proportional} to $a$, $b$, and
$a^{\prime}$.
\end{definition}

\begin{theorem}
\label{fourth proportional unique}If there is a fourth proportional to $a$,
$b$, and $a^{\prime}$, then it is unique.
\end{theorem}

\begin{proof}
Suppose $b^{\prime}$ and $c^{\prime}$ are each fourth proportionals to $a$,
$b$, and $a^{\prime}$. Then $a:b=a^{\prime}:b^{\prime}$ and $a:b=a^{\prime
}:c^{\prime}$ by Definition \ref{definition fourth proportional}. Hence
$a^{\prime}:b^{\prime}=a^{\prime}:c^{\prime}$ by Proposition \ref{V 11} and
$b^{\prime}=c^{\prime}$ by Proposition \ref{V 9}.
\end{proof}

\begin{theorem}
If $\varphi:M\rightarrow M^{\prime}$ is an embedding which maps $a$ into
$a^{\prime}$, then $\varphi b$ is the fourth proportional to $a,b,a^{\prime} $
for each $b$.
\end{theorem}

\begin{proof}
Theorem \ref{homomorphisms are proportional} and Definition
\ref{definition fourth proportional}.
\end{proof}

\begin{theorem}
\label{embedding from fourth proportional}If $a$ and $a^{\prime}$ are fixed
and there is, for each $b$, a fourth proportional to $a,b,a^{\prime}$, then
there is an embedding $\varphi:M\rightarrow M^{\prime}$ which maps $a$ into
$a^{\prime}$.
\end{theorem}

\begin{proof}
Let $a$ and $a^{\prime}$ be fixed and suppose that there is, for each $b$, a
fourth proportional to $a,b,a^{\prime}$. Then for each $b$ there is exactly
one fourth proportional to $a,b,a^{\prime}$ by Theorem
\ref{fourth proportional unique}. Thus we can define a function $\varphi
:M\rightarrow M^{\prime}$ such that $\varphi b$ is the fourth proportional to
$a$, $b$, and $a^{\prime}$.

I say $\varphi$ is an embedding. For any two elements $b$ and $c$ of $M$, we
have by assumption%
\[
a:b=a^{\prime}:\varphi b\text{, \ }a:c=a^{\prime}:\varphi c\text{, \ and
}a:b+c=a^{\prime}:\varphi\left(  b+c\right)
\]
or equivalently by Definition \ref{revised euclid def of ratios}%
\[
\varphi b:a^{\prime}=b:a\text{, \ }\varphi c:a^{\prime}=c:a\text{, \ and
}\varphi\left(  b+c\right)  :a^{\prime}=b+c:a\text{.}%
\]
Thus%
\[
\varphi b+\varphi c:a^{\prime}=b+c:a
\]
by Proposition \ref{V 24},%
\[
\varphi\left(  b+c\right)  :a^{\prime}=\varphi b+\varphi c:a^{\prime}%
\]
by Proposition \ref{V 11}, and
\[
\varphi\left(  b+c\right)  =\varphi b+\varphi c
\]
by Proposition \ref{V 9}$.$ We have now shown that $\varphi$ is homomorphism
according to Definition \ref{definition homomorphism} and hence $\varphi$ is
an embedding by Theorem \ref{morphisms 1-1}. And $\varphi a$ is the fourth
proportional to $a$, $a$, and $a^{\prime}$, or in other words $a:a=a^{\prime
}:\varphi a$. And since $1a=1a$, it follows that $1a^{\prime}=1\left(  \varphi
a\right)  $ by Definition \ref{revised euclid def of ratios}. Therefore
$a^{\prime}=\varphi a$ by Definition \ref{definition integral multiple}.
\end{proof}

\ 

In fact the embedding constructed in the preceding theorem is the unique
embedding of $M$ into $M^{\prime}$ which maps $a$ into $a^{\prime}$. This and
a bit more is established in the next theorem.

\begin{theorem}
\label{unique morphism}If $\varphi$ and $\chi$ are embeddings from $M$ into
$M^{\prime}$, then $\varphi a$ has to $\chi a$ the same relation (%
$<$%
, =, or
$>$%
) as $\varphi b$ has to $\chi b$. (In particular if $\varphi a=\chi a$ for one
$a$, then $\varphi b=\chi b$ for all $b$ and $\varphi=\chi$.)
\end{theorem}

\begin{proof}
Assume $\varphi$ and $\chi$ are embeddings from $M$ into $M^{\prime}$. Then%
\[
\varphi a:\varphi b=a:b\text{ and }\chi a:\chi b=a:b
\]
by Theorem \ref{homomorphisms are proportional} and hence%
\[
\varphi a:\varphi b=\chi a:\chi b
\]
by Proposition \ref{V 11}. Therefore $\varphi a$ has to $\chi a$ the same
relation (%
$<$%
, =, or
$>$%
) as $\varphi b$ has to $\chi b $ by Proposition \ref{V 14}.
\end{proof}

\begin{theorem}
\label{endomorphisms commute}If $\varphi$ and $\chi$ are endomorphisms of $M$,
then $\varphi\chi=\chi\varphi$.
\end{theorem}

\begin{proof}
Assume $\varphi$ and $\chi$ are endomorphisms of $M$. Then for any $a\in M$,%
\[
\varphi\left(  \chi a\right)  :\varphi a=\chi a:a\text{ and }\varphi
a:a=\chi\left(  \varphi a\right)  :\chi a
\]
by Theorem \ref{homomorphisms are proportional},%
\[
\varphi\left(  \chi a\right)  :a=\chi\left(  \varphi a\right)  :a
\]
by Proposition \ref{V 23}, and%
\[
\varphi\left(  \chi a\right)  =\chi\left(  \varphi a\right)
\]
by Proposition \ref{V 9}. Therefore $\left(  \varphi\chi\right)  a=\left(
\chi\varphi\right)  a$ by Definition \ref{definition composition} and
$\varphi\chi=\chi\varphi$ by Definition\ \ref{definition equal functions}.
\end{proof}

\section{Continuous Magnitude Spaces}

In this section we come to the central theorem: If $a$ is an element of an
Archimedean magnitude space $M$ and $a^{\prime}$ is an element of a continuous
magnitude space $M^{\prime}$, then there is a unique embedding of $M$ into
$M^{\prime}$ which maps $a$ into $a^{\prime}$. Before tackling the main
theorem we need to establish some results concerning ratios which, although
very basic, have not been required prior to this section.

\begin{theorem}
If $a:b>a^{\prime}:b^{\prime}$ and $ma\leq nb$, then $ma^{\prime}<nb^{\prime}$.
\end{theorem}

\begin{proof}
$ma:mb=a:b$ by Proposition \ref{V 15} and $a:b>a^{\prime}:b^{\prime}$ by
assumption. Hence $ma:mb>a^{\prime}:b^{\prime}$ by Proposition \ref{V 13}.
Likewise, $a^{\prime}:b^{\prime}=ma^{\prime}:mb^{\prime}$ by Proposition
\ref{V 15} and hence $ma:mb>ma^{\prime}:mb^{\prime}$ by Proposition
\ref{V 13}. Thus there exist $j$ and $k$ such that
\[
j\left(  ma\right)  >k\left(  mb\right)  \text{ and }j\left(  ma^{\prime
}\right)  \leq k\left(  mb^{\prime}\right)
\]
by Definition \ref{revised euclid def of ratios}. And by assumption, $ma\leq
nb$ and so $j\left(  ma\right)  \leq j\left(  nb\right)  $ by Proposition
\ref{V 5}. And also, from above, $j\left(  ma\right)  >k\left(  mb\right)  $
and hence $k\left(  mb\right)  <j\left(  nb\right)  $ by Theorem
\ref{Transitivity mixed} and hence $km<jn$ by Propositions \ref{V 3}\ and
\ref{V 6}. Hence $k\left(  mb^{\prime}\right)  <j\left(  nb^{\prime}\right)  $
by Propositions \ref{V 6} and \ref{V 3}. And also, from above, $j\left(
ma^{\prime}\right)  \leq k\left(  mb^{\prime}\right)  $ and hence $j\left(
ma^{\prime}\right)  <j\left(  nb^{\prime}\right)  $ by Theorem
\ref{Transitivity mixed} and hence $ma^{\prime}<nb^{\prime}$ by Proposition
\ref{V 5}; the very thing to be shown.
\end{proof}

\begin{theorem}
\label{Transitivity of ratios}If $a:b>a^{\prime}:b^{\prime}$ and $a^{\prime
}:b^{\prime}>a^{\prime\prime}:b^{\prime\prime}$, then $a:b>a^{\prime\prime
}:b^{\prime\prime}$.
\end{theorem}

\begin{proof}
Assume $a:b>a^{\prime}:b^{\prime}$ and $a^{\prime}:b^{\prime}>a^{\prime\prime
}:b^{\prime\prime}$. From the first ratio inequality there are natural numbers
$m$ and $n$ such that $ma>nb$ and $ma^{\prime}\leq nb^{\prime}$ by Definition
\ref{revised euclid def of ratios}. And also $ma^{\prime\prime}<nb^{\prime
\prime}$ by the preceding theorem. Hence $ma>nb$ and $ma^{\prime\prime
}<nb^{\prime\prime}$ and therefore $a:b>a^{\prime\prime}:b^{\prime\prime}$
according to Definition \ref{revised euclid def of ratios}.
\end{proof}

\begin{theorem}
If there are natural numbers $j,k$ such that $ja>kb$ and $ja^{\prime
}=kb^{\prime}$, then there are natural numbers $m,n$ such that $ma>nb$ and
$ma^{\prime}<nb^{\prime}$.
\end{theorem}

\begin{proof}
Assume $ja>kb$ and $ja^{\prime}=kb^{\prime}$. Since magnitude spaces are
assumed to be Archimedean, there is some natural number $p$ such that
$p\left(  ja-kb\right)  >1a$ by Theorem \ref{some n such that na>b}. Therefore%
\begin{align*}
p\left(  ja-kb\right)  >a  & \implies p\left(  ja\right)  -p\left(  kb\right)
>1a\text{ (Proposition \ref{V 5})}\\
& \implies\left(  pj\right)  a-\left(  pk\right)  b>1a\text{ (Proposition
\ref{V 3})}\\
& \implies\left(  pj\right)  a>\left(  pk\right)  b+1a\text{ (Theorem
\ref{moving b accross <=>})}\\
& \implies\left(  pj\right)  a-1a>\left(  pk\right)  b\text{ (Theorem
\ref{moving b accross <=>})}\\
& \implies\left(  pj-1\right)  a>\left(  pk\right)  b\text{. (Theorem
\ref{1a and (n+1)a})}%
\end{align*}
And since$\ ja^{\prime}=kb^{\prime}$,%
\begin{align*}
\left(  pj-1\right)  a^{\prime}  & <\left(  pj\right)  a^{\prime}\text{
(Theorem \ref{b-a < a} and Proposition \ref{V 6})}\\
& =p\left(  ja^{\prime}\right)  \text{ (Proposition \ref{V 3})}\\
& =p\left(  kb^{\prime}\right)  \text{ (}ja^{\prime}=kb^{\prime}\text{)}\\
& =\left(  pk\right)  b^{\prime}\text{ ((Proposition \ref{V 3})}%
\end{align*}
and hence $\left(  pj-1\right)  a^{\prime}<\left(  pk\right)  b^{\prime}$. We
now have two natural numbers $m=pj-1$ and $n=pk$ such that $ma>nb$ and
$ma^{\prime}<nb^{\prime}$.
\end{proof}

\begin{theorem}
\label{unequal ratios}If $a:b\neq a^{\prime}:b^{\prime}$, then $a:b>a^{\prime
}:b^{\prime}$ or $a^{\prime}:b^{\prime}>a:b$.
\end{theorem}

\begin{proof}
If $a:b\neq a^{\prime}:b^{\prime}$, then there are two natural numbers $j,k $
for which at least one of following six cases is true:%

\begin{align*}
\text{1) }ja<kb\text{ }  & \text{and}\text{ }ja^{\prime}=kb^{\prime}\\
\text{2) }ja<kb\text{ }  & \text{and}\text{ }ja^{\prime}>kb^{\prime}\\
\text{3) }ja=kb\text{ }  & \text{and}\text{ }ja^{\prime}<kb^{\prime}\\
\text{4) }ja=kb\text{ }  & \text{and}\text{ }ja^{\prime}>kb^{\prime}\\
\text{5) }ja>kb\text{ }  & \text{and}\text{ }ja^{\prime}<kb^{\prime}\\
\text{6) }ja>kb\text{ }  & \text{and}\text{ }ja^{\prime}=kb^{\prime}%
\end{align*}
In each of the cases 2, 4, \ 5, and 6, $a:b>a^{\prime}:b^{\prime}$ or
$a^{\prime}:b^{\prime}>a:b$ according to Definition
\ref{revised euclid def of ratios}. In case 1, there are two natural numbers
$m,n$ such that $na<mb$ and $na^{\prime}>mb^{\prime}$ by the preceding theorem
and hence $a^{\prime}:b^{\prime}>a:b$. In a similar fashion we can show that
in case 3 $a:b>a^{\prime}:b^{\prime}$.
\end{proof}

\begin{theorem}
\label{strong form of nondiscrete}If $M$ is nondiscrete, then for any $a\in M
$ and $n\in%
\mathbb{N}
$ there is some $b\in M$ such that $nb<a$.
\end{theorem}

\begin{proof}
Assume $M$ is nondiscrete. Then the theorem is true for $n=1$ since there is
some element $b<a$ by Definition \ref{definition discrete} and $1b=b$ by
Definition \ref{definition integral multiple}. Now suppose the theorem is true
for $n$ and $nb<a.$ There is some $c\in M$ such that $c<b$ by Definition
\ref{definition discrete}. Let $d$ be the smaller of $\left(  b-c\right)  $
and $c$ so that $d\leq b-c$ and $d\leq c$. There is some $e\in M $ such that
$e<d$ by Definition \ref{definition discrete} and hence $e<b-c$ and $e<c$ by
Theorem \ref{Transitivity mixed}. Hence
\begin{align*}
e+e  & <\left(  b-c\right)  +e\text{ (}e<b-c\text{ and Theorem
\ref{equals added to equals or unequals})}\\
& <\left(  b-c\right)  +c\text{ (}e<c\text{ and Theorem
\ref{equals added to equals or unequals})}\\
& =b\text{. (Definition \ref{definition subtraction})}%
\end{align*}
And hence $e+e<b$ by Theorem \ref{Transitivity of >}. Now $1\leq n$ by
Definition \ref{definition natural numbers} and hence $n+1\leq n+n$ by Theorem
\ref{equals added to equals or unequals}. Therefore
\begin{align*}
\left(  n+1\right)  e  & \leq\left(  n+n\right)  e\text{ (Proposition
\ref{V 6})}\\
& =ne+ne\text{ (Proposition \ref{V 2})}\\
& =n\left(  e+e\right)  \text{ (Proposition \ref{V 1})}\\
& <nb\text{ (Proposition \ref{V 5})}\\
& <a\text{ (assumption)}%
\end{align*}
and so $\left(  n+1\right)  e<a$ by Theorem \ref{Transitivity mixed} and hence
the theorem is true for $n+1$. Therefore the theorem is true for all $n$ by
Theorem \ref{induction}.
\end{proof}

\begin{theorem}
\label{continuity of ratios}If $M$ and $M^{\prime}$ are magnitude spaces, $M $
is nondiscrete, $a,b\in M$, $a^{\prime},b^{\prime}\in M^{\prime}$, then if
$a:b>a^{\prime}:b^{\prime}$ there exists some $c\in M$ such that
$a:b>c:b>a^{\prime}:b^{\prime}$, and if $a:b<a^{\prime}:b^{\prime}$ there
exists some $c\in M$ such that $a:b<c:b<a^{\prime}:b^{\prime}$.
\end{theorem}

\begin{proof}
If $a:b>a^{\prime}:b^{\prime}$, then there are $m,n\in%
\mathbb{N}
$ such that
\[
ma>nb\text{ and }ma^{\prime}\leq nb^{\prime}%
\]
by Definition \ref{revised euclid def of ratios}. From the preceding theorem,
there is some $d\in M$ such that%
\[
ma-nb>md.
\]
And%
\begin{align*}
ma-nb>md  & \implies ma>md+nb\text{ (Theorem \ref{moving b accross <=>})}\\
& \implies ma-md>nb\text{ (Theorem \ref{moving b accross <=>})}\\
& \implies m\left(  a-d\right)  >nb\text{. (Proposition \ref{V 5})}%
\end{align*}
Letting $c=a-d$ we have
\[
mc>nb\text{ and }ma^{\prime}\leq nb^{\prime}%
\]
and hence $c:b>a^{\prime}:b^{\prime}$. And since $a>c$, from Proposition
\ref{V 8}, $a:b>c:b$. The second part of the theorem is proved in a similar manner.
\end{proof}

\ 

With these preliminaries out of the way, we are ready for the main theorem in
this section.

\begin{theorem}
If $M$ is a magnitude space, $M^{\prime}$ is a continuous magnitude space,
$a\in M$, and $a^{\prime}\in M^{\prime}$, then for each $b\in M$ there is a
fourth proportional $b^{\prime}\in M^{\prime}$ to $a$, $b$, and $a^{\prime}$.
\end{theorem}

\begin{proof}
Assume $M^{\prime}$ is continuous and hence complete and nondiscrete by
Definition \ref{definition continuous}. Let $a,b\in M$ and $a^{\prime}\in
M^{\prime}$ be given. Let%
\[
A=\left\{  c^{\prime}\in M^{\prime}\mid c^{\prime}:a^{\prime}<b:a\right\}
\text{ and }B=\left\{  c^{\prime}\in M^{\prime}\mid c^{\prime}:a^{\prime
}>b:a\right\}  \text{.}%
\]
We first show that each of these sets is nonempty. Since $M$ is Archimedean,
there is some natural number $m$ such that $mb>1a$ by Theorem
\ref{some n such that na>b}. And since $M^{\prime}$ is nondiscrete, there is
some $c^{\prime}\in M^{\prime}$ such that $mc^{\prime}<1a^{\prime}$ from
Theorem \ref{strong form of nondiscrete}. Then $mb>1a$ and $mc^{\prime
}<1a^{\prime}$; hence $b:a>c^{\prime}:a^{\prime}$ by Definition
\ref{revised euclid def of ratios} and $c^{\prime}\in A$. Likewise, there is
some natural number $n$ such that $1b<na$ by Theorem
\ref{some n such that na>b}. Let $c^{\prime}=na^{\prime}+a^{\prime}$. Then
$1c^{\prime}>na^{\prime}$ and $1b<na$; hence $c^{\prime}:a^{\prime}>b:a$ by
Definition \ref{revised euclid def of ratios} and $c^{\prime}\in B$.

If $c^{\prime}\in A$ and $d^{\prime}\in B$, then $c^{\prime}:a^{\prime}<b:a$
and $b:a<d^{\prime}:a^{\prime}$. Thus $c^{\prime}:a^{\prime}<d^{\prime
}:a^{\prime}$ from Theorem \ref{Transitivity of ratios} and hence $c^{\prime
}<d^{\prime}$ from Proposition \ref{V 10}. Thus every element of $B$ is an
upper bound of $A$ and every element of $A$ is a lower bound of $B$ according
to Definition \ref{definition bounds}.

Since $M^{\prime}$ is complete, $A$ has a least upper bound $b^{\prime}$ by
Definition \ref{definition complete}. I say $b^{\prime}$ is a fourth
proportional to $a$, $b$, and $a^{\prime}$. That is, $a:b=a^{\prime}%
:b^{\prime}$ or equivalently $b^{\prime}:a^{\prime}=b:a$. Suppose that
$b^{\prime}$ is not a fourth proportional to $a$, $b$, and $a^{\prime}$. Then
$b^{\prime}:a^{\prime}<b:a$ or $b^{\prime}:a^{\prime}>b:a$ from Theorem
\ref{unequal ratios}, or what is the same $b^{\prime}\in A$ or $b^{\prime}\in
B$.

If $b^{\prime}\in A$ (i.e. $b^{\prime}:a^{\prime}<b:a$), then there is a
$c^{\prime}\in M^{\prime}$ such that $b^{\prime}:a^{\prime}<c^{\prime
}:a^{\prime}<b:a$ from Theorem \ref{continuity of ratios} and hence
$c^{\prime}\in A$ and $b^{\prime}<c^{\prime}$. But $b^{\prime}$ is an upper
bound for $A$ which is a contradiction. Therefore $b^{\prime}\notin A$.

If $b^{\prime}\in B$ (i.e. $b^{\prime}:a^{\prime}>b:a$), then from Theorem
\ref{continuity of ratios} there is a $c^{\prime}\in M^{\prime}$ such that
$b^{\prime}:a^{\prime}>c^{\prime}:a^{\prime}>b:a$ and hence $c^{\prime}\in B$
and $b^{\prime}>c^{\prime}$. But since $c^{\prime}\in B$, $c^{\prime}$ is an
upper bound of $A$ and it is smaller than the least upper bound of $A$ which
is a contradiction. Therefore $b^{\prime}\notin B$.

We have now shown that $b^{\prime}$ is not an element of $A$ or $B$ which
contradicts $b^{\prime}$ not being the fourth proportional to $a$, $b$, and
$a^{\prime}$. Therefore, indeed, $b^{\prime}$ is the fourth proportional to
$a$, $b$, and $a^{\prime}$.
\end{proof}

\begin{theorem}
\label{complete f(a)=b}If $M$ is a magnitude space, $M^{\prime}$ is a
continuous magnitude space, $a\in M$, and $a^{\prime}\in M^{\prime}$, then
there exists a unique embedding $\varphi:M\rightarrow M^{\prime}$ such that
$\varphi a=a^{\prime}$.
\end{theorem}

\begin{proof}
Preceding theorem, Theorem \ref{embedding from fourth proportional}, and
Theorem \ref{unique morphism}.
\end{proof}

\begin{theorem}
\label{complete embedding is onto}If $M$ and $M^{\prime}$ are continuous
magnitude spaces, $a\in M$ and $a^{\prime}\in M^{\prime}$, then there exists a
unique isomorphism $\varphi:M\rightarrow M^{\prime}$ such that $\varphi
a=a^{\prime}$.
\end{theorem}

\begin{proof}
Since $M^{\prime}$ is continuous, there exists a unique embedding
$\varphi:M\rightarrow M^{\prime}$ such that $\varphi a=a^{\prime}$ by the
preceding theorem. And since $M$ is continuous, there is also a unique
embedding $\chi:M^{\prime}\rightarrow M$ such that $\chi a^{\prime}=a$ by the
preceding theorem. But then $\varphi\circ\chi$ is an embedding of $M^{\prime}$
into $M^{\prime}$ by Theorem \ref{composition of homomophisms}. And
$\varphi\circ\chi$ maps $a^{\prime}$ into $a^{\prime}$. But the identity map
$i_{M^{\prime}}$ on $M^{\prime}$ is an embedding by Theorem
\ref{identity homomorphism} and also $i_{M^{\prime}}$ maps $a^{\prime}$ into
$a^{\prime}$. Hence $\varphi\circ\chi=i_{M^{\prime}}$ by Theorem
\ref{unique morphism}. Hence $\varphi$ is onto by Theorem
\ref{proving function is onto} and therefore is an isomorphism according to
Definition \ref{definition isomorphism isomorphic}.
\end{proof}

\begin{theorem}
\label{continuous magnitude spaces isomorhpic}Any two continuous magnitude
spaces are isomorphic.
\end{theorem}

\begin{proof}
Let $M$ and $M^{\prime}$ be any two continuous magnitude spaces. Pick any
$a\in M$ and $a^{\prime}\in M^{\prime}$. There is an embedding $\varphi
:M\rightarrow M^{\prime}$ such that $\varphi a=a^{\prime}$ by Theorem
\ref{complete f(a)=b}. The embedding $\varphi$ is an isomorphism by the
preceding theorem and therefore $M$ is isomorphic to $M^{\prime}$ according to
Definition \ref{definition isomorphism isomorphic}.
\end{proof}

\newpage

\section{Real Numbers}

Let us now review the definition of the natural numbers and define the
positive real numbers.

\begin{enumerate}
\item Well ordered magnitude spaces are complete (Theorem
\ref{well ordered implies discrete and complete}) and hence Archimedean
(Theorem \ref{Holder}).

\item If $M$ is a well ordered magnitude space, $M^{\prime}$ is any magnitude
space (not necessarily Archimedean), $a$ is the smallest element in $M$, and
$a^{\prime}$ is any element in $M^{\prime}$, then there exists a unique
embedding $\varphi:M\rightarrow M^{\prime}$ such that $\varphi a=a^{\prime}$
(Theorem \ref{embedding discrete magnitude}).

\item Any two well ordered magnitude spaces are isomorphic (Theorem
\ref{well ordered magnitude spaces isomorphic}).
\end{enumerate}

Because a well ordered magnitude space can be embedded into any magnitude
space, we say that well ordered magnitude spaces are \textbf{minimal}
magnitude spaces. And for the same reason we also say that the well ordered
magnitude spaces are minimal Archimedean magnitude spaces. We defined the
natural numbers as an arbitrary representative of the well ordered magnitude spaces.

\begin{enumerate}
\item Continuous magnitude spaces are complete (Definition
\ref{definition continuous}) and hence Archimedean (Theorem \ref{Holder}).

\item If $M$ is an Archimedean magnitude space, $M^{\prime}$ is a continuous
magnitude space, $a\in M$, and $a^{\prime}\in M^{\prime}$, then there exists a
unique embedding $\varphi:M\rightarrow M^{\prime}$ such that $\varphi
a=a^{\prime}$ (Theorem \ref{complete f(a)=b}).

\item Any two continuous magnitude spaces are isomorphic (Theorem
\ref{continuous magnitude spaces isomorhpic}).
\end{enumerate}

Because any Archimedean magnitude space can be embedded into a continuous
magnitude space, we say that continuous magnitudes are \textbf{maximal}
Archimedean magnitude spaces. We now define the positive real numbers as an
arbitrary representative of the continuous magnitude spaces.

\begin{definition}
\label{definition real number}Pick any continuous magnitude space and denote
it by $%
\mathbb{R}
_{+}$ and pick any element of $%
\mathbb{R}
_{+}$ and denote it by $1$. Since all continuous magnitude spaces are
isomorphic (Theorem \ref{continuous magnitude spaces isomorhpic}) and we will
use only the algebraic properties that the continuous magnitude spaces have in
common, it does not matter which continuous magnitude space is chosen to play
the role of the \textquotedblleft number system\textquotedblright\ $%
\mathbb{R}
_{+}$. We call $%
\mathbb{R}
_{+}$ the \textbf{positive real numbers}.
\end{definition}

When we defined the natural numbers, we immediately defined an integral
multiple $na$ where $n$ is a natural number and $a$ is an element of a
magnitude space. The remainder of this work examines similar constructions of
multiples where the multiplier is not necessarily a natural number. This will
lead us to real multiples of real numbers which is the familiar binary product
operator in $%
\mathbb{R}
_{+}$.

\section{Magnitude Embedding Spaces\label{sec embedding spaces}}

Definitions of multiples and products are based on embeddings of one magnitude
space into another. It is useful at the onset to consider the general case.

\begin{definition}
$H\left(  M,M^{\prime}\right)  $ is the set of all embeddings from $M$ into
$M^{\prime}$.
\end{definition}

It may happen that $H\left(  M,M^{\prime}\right)  $ is the empty set. For
example, if $M$ is nondiscrete and $M^{\prime}$ is discrete, there are no
embeddings of $M$ into $M^{\prime}$. If $H\left(  M,M^{\prime}\right)  $ is
not empty and $\varphi,\chi\in H\left(  M,M^{\prime}\right)  $, then we have
already shown (Theorem \ref{sum of homomorphisms}) that $\varphi+\chi\in
H\left(  M,M^{\prime}\right)  $.

\begin{theorem}
\label{H(M,N) is a magnitude}If $H\left(  M,M^{\prime}\right)  $ is nonempty,
then $H\left(  M,M^{\prime}\right)  $ is a magnitude space.
\end{theorem}

\begin{proof}
Let $\varphi,\chi,\psi\in H\left(  M,M^{\prime}\right)  $.

(i)For any $a\in M$%
\begin{align*}
\left(  \varphi+\left(  \chi+\psi\right)  \right)  a  & =\varphi a+\left(
\chi+\psi\right)  a\text{ (Definition \ref{definition sum of embeddings})}\\
& =\varphi a+\left(  \chi a+\psi a\right)  \text{ (Definition
\ref{definition sum of embeddings})}\\
& =\left(  \varphi a+\chi a\right)  +\psi a\text{ (Definition
\ref{definition magnitude})}\\
& =\left(  \varphi+\chi\right)  a+\psi a\text{ (Definition
\ref{definition sum of embeddings})}\\
& =\left(  \left(  \varphi+\chi\right)  +\psi\right)  a\text{ (Definition
\ref{definition sum of embeddings})}%
\end{align*}
and hence $\varphi+\left(  \chi+\psi\right)  =\left(  \varphi+\chi\right)
+\psi$ according to Definition \ref{definition equal functions}.

(ii)For any $a\in M$%
\begin{align*}
\left(  \varphi+\chi\right)  a  & =\varphi a+\chi a\text{ (Definition
\ref{definition sum of embeddings})}\\
& =\chi a+\varphi a\text{ (Definition \ref{definition magnitude})}\\
& =\left(  \chi+\varphi\right)  a\text{ (Definition
\ref{definition sum of embeddings})}%
\end{align*}
and hence $\varphi+\chi=\chi+\varphi$ according to Definition
\ref{definition equal functions}.

(iii)Fix $a\in M$. Since $\varphi a,\chi a\in M^{\prime}$, exactly one of the
following is true:%
\begin{align*}
\varphi a  & =\chi a\text{,}\\
\varphi a  & >\chi a\text{, or}\\
\chi a  & >\varphi a
\end{align*}
by Theorem \ref{< is trichotomous}. In the first case, Theorem
\ref{unique morphism} shows that $\varphi b=\chi b$ for all $b\in M$ and hence
$\varphi=\chi$ according to Definition \ref{definition equal functions}. In
the second case, $\varphi b>\chi b$ for all $b\in M$ by Theorem
\ref{unique morphism}. We can then define a function $d:M\rightarrow
M^{\prime}$ \ by%
\[
\delta b=\varphi b-\chi b\text{.}%
\]
For any $b,c\in M$%
\begin{align*}
\varphi\left(  b+c\right)   & =\varphi b+\varphi c\text{ (Definition
\ref{definition homomorphism})}\\
& =\left(  \left(  \varphi b-\chi b\right)  +\chi b\right)  +\left(  \left(
\varphi c-\chi c\right)  +\chi c\right)  \text{ (Definition
\ref{definition subtraction})}\\
& =\left(  \left(  \varphi b-\chi b\right)  +\left(  \varphi c-\chi c\right)
\right)  +\left(  \chi b+\chi c\right)  \text{ (associativity and
commutativity of }+\text{)}\\
& =\left(  \delta b+\delta c\right)  +\left(  \chi b+\chi c\right)  \text{
(definition of }\delta\text{)}\\
& =\left(  \delta b+\delta c\right)  +\chi\left(  b+c\right)  \text{
(Definition \ref{definition homomorphism})}%
\end{align*}
and hence $\varphi\left(  b+c\right)  -\chi\left(  b+c\right)  =\left(  \delta
b+\delta c\right)  $ by Definition \ref{definition subtraction}. But
$\delta\left(  b+c\right)  =\varphi\left(  b+c\right)  -\chi\left(
b+c\right)  $ by the definition of $\delta$. Hence $\delta\left(  b+c\right)
=\delta b+\delta c$ and $\delta$ is an element of $H\left(  M,M^{\prime
}\right)  $ by Definition \ref{definition homomorphism} and Theorem
\ref{morphisms 1-1}. And, of course,%
\begin{align*}
\varphi\left(  b\right)   & =\chi b+\left(  \varphi b-\chi b\right)  \text{
(Definition \ref{definition subtraction})}\\
& =\chi b+\delta b\text{ (definition of }\delta\text{)}\\
& =\left(  \chi+\delta\right)  b\text{ (Definition
\ref{definition sum of embeddings})}%
\end{align*}
for all $b\in M$ and hence $\varphi=\chi+\delta$ according to Definition
\ref{definition equal functions}. A similar argument applies in the third case
where $\chi a>\varphi a$. So, in summary, if $\varphi,\chi\in H\left(
M,M^{\prime}\right)  $, then exactly one of the following is true:%
\begin{align*}
\varphi & =\chi\text{,}\\
\varphi & =\chi+\delta\text{ for some }\delta\in H\left(  M,M^{\prime}\right)
\text{, or}\\
\chi & =\varphi+\delta\text{ for some }\delta\in H\left(  M,M^{\prime}\right)
\text{.}%
\end{align*}
We have now shown that $H\left(  M,M^{\prime}\right)  $, with the usual
addition of embeddings, is a magnitude space according to Definition
\ref{definition magnitude}.
\end{proof}

\section{Magnitude Endomorphism Spaces\label{sec endomorphism spaces}}

A special case of magnitude space embeddings are endomorphisms; that is
embeddings of a magnitude space into itself.

\begin{definition}
If $M$ is a magnitude space, then we denote the set of endomorphisms of $M$ by
$E\left(  M\right)  $.
\end{definition}

Of course $E\left(  M\right)  =H\left(  M,M\right)  $. Now in general
$H\left(  M,M^{\prime}\right)  $ can be empty. But $E\left(  M\right)  $
always has elements. In fact for each natural number $n$ the mapping
$\varphi:M\rightarrow M$ defined by $\varphi a=na$ is an embedding of $M$ into
$M$ and hence is an element of $E\left(  M\right)  $.

If $\varphi,\chi\in E\left(  M\right)  $ then $\varphi+\chi$ and $\varphi
\circ\chi$ are each elements of $E\left(  M\right)  $ by Theorems
\ref{sum of homomorphisms} and \ref{composition of homomophisms}. We have
already shown that $E\left(  M\right)  =H\left(  M,M\right)  $ is a magnitude
space with the $+$ operator on elements of $E\left(  M\right)  $ (Theorem
\ref{H(M,N) is a magnitude}). Also we know that the composition operator
$\circ$ is commutative (Theorem \ref{endomorphisms commute}) and associative
(Theorem \ref{composition is associative}).

\begin{theorem}
\label{composition distributes over addition}If $\varphi,\chi,\psi\in E\left(
M\right)  $ then $\varphi\circ\left(  \chi+\psi\right)  =\varphi\circ
\chi+\varphi\circ\psi$ and $\left(  \varphi+\chi\right)  \circ\psi
=\varphi\circ\psi+\chi\circ\psi$.
\end{theorem}

\begin{proof}
If $a\in M$ then%
\begin{align*}
\left(  \varphi\circ\left(  \chi+\psi\right)  \right)  a  & =\varphi\left(
\left(  \chi+\psi\right)  a\right)  \text{ (Definition
\ref{definition composition})}\\
& =\varphi\left(  \chi a+\psi a\right)  \text{ (Definition
\ref{definition sum of embeddings})}\\
& =\varphi\left(  \chi a\right)  +\varphi\left(  \psi a\right)  \text{
(Definition \ref{definition homomorphism})}\\
& =\left(  \varphi\circ\chi\right)  a+\left(  \varphi\circ\psi\right)  a\text{
(Definition \ref{definition composition})}\\
& =\left(  \varphi\circ\chi+\varphi\circ\psi\right)  a\text{ (Definition
\ref{definition sum of embeddings})}%
\end{align*}
and hence $\varphi\circ\left(  \chi+\psi\right)  =\varphi\circ\chi
+\varphi\circ\psi$ by Definition \ref{definition equal functions}. The second
part of the theorem can be proved in a similar manner.
\end{proof}

\begin{theorem}
\label{composition preserves order}If $\varphi,\chi,\psi\in E\left(  M\right)
$ then $\varphi\circ\chi$ has to $\varphi\circ\psi$ the same relation (%
$<$%
, =, or
$>$%
) as $\chi$ has to $\psi$ and $\chi\circ\varphi$ has to $\psi\circ\varphi$ the
same relation (%
$<$%
, =, or
$>$%
) as $\chi$ has to $\psi$.
\end{theorem}

\begin{proof}
Fix $\varphi\in E\left(  M\right)  $ and let the mapping $\Psi:E\left(
M\right)  \rightarrow E\left(  M\right)  $ be defined by $\Psi\chi
=\varphi\circ\chi$. If $\chi,\psi\in E\left(  M\right)  $, then
\begin{align*}
\Psi\left(  \chi+\psi\right)   & =\varphi\circ\left(  \chi+\psi\right)  \text{
(definition of }\Psi\text{)}\\
& =\varphi\circ\chi+\varphi\circ\psi\text{ (Theorem
\ref{composition distributes over addition})}\\
& =\Psi\chi+\Psi\psi\text{ (definition of }\Psi\text{)}%
\end{align*}
and hence $\Psi$ is a homomorphism by Definition \ref{definition homomorphism}%
. Therefore $\varphi\circ\chi=\Psi\chi$ has to $\varphi\circ\psi=\Psi\psi$ the
same relation (%
$<$%
, =, or
$>$%
) as $\chi$ has to $\psi$ (and $\Psi$ is an embedding) by Theorem
\ref{morphisms 1-1}. The second part of the theorem can be proved in a similar manner.
\end{proof}

\begin{theorem}
\label{composition identity element}If $\varphi\in E\left(  M\right)  $, then
$\varphi\circ i_{M}=i_{M}\circ\varphi=\varphi$.
\end{theorem}

\begin{proof}
If $a\in M$, then%
\begin{align*}
\left(  \varphi\circ i_{M}\right)  a  & =\varphi\left(  i_{M}a\right)  \text{
(Definition \ref{definition composition})}\\
& =\varphi a\text{ (Definition \ref{definition identity function})}%
\end{align*}
and therefore $\varphi\circ i_{M}=\varphi$ by Definition
\ref{definition equal functions}. And $\varphi\circ i_{M}=i_{M}\circ\varphi$
since endomorphisms commute by Theorem \ref{endomorphisms commute}.
\end{proof}

At this point we would like to point out that the magnitude space $E\left(
M\right)  $ for an arbitrary (Archimedean) magnitude space has exactly those
properties that we associate with addition and multiplication of positive numbers.

\begin{enumerate}
\item $E\left(  M\right)  $ with the addition operator is a magnitude space
(Theorem \ref{H(M,N) is a magnitude}).

\begin{enumerate}
\item $\varphi+\left(  \chi+\psi\right)  =\left(  \varphi+\chi\right)  +\psi$

\item $\varphi+\chi=\chi+\varphi$

\item Exactly one of the following is true: $\varphi=\chi$, or $\varphi
=\chi+\delta$ for some $\delta\in E\left(  M\right)  $, or $\chi
=\varphi+\delta$ for some $\delta\in E\left(  M\right)  $.
\end{enumerate}

\item The composition operator is associative and commutative (Theorems
\ref{composition is associative} and \ref{endomorphisms commute}).

\begin{enumerate}
\item $\psi\circ\left(  \chi\circ\varphi\right)  =\left(  \psi\circ
\chi\right)  \circ\varphi$

\item $\varphi\circ\chi=\chi\circ\varphi$
\end{enumerate}

\item Composition distributes over addition (Theorem
\ref{composition distributes over addition}).

\begin{enumerate}
\item $\varphi\circ\left(  \chi+\psi\right)  =\varphi\circ\chi+\varphi
\circ\psi$

\item $\left(  \varphi+\chi\right)  \circ\psi=\varphi\circ\psi+\chi\circ\psi$
\end{enumerate}

\item There is an identity element for the composition operator (Theorem
\ref{composition identity element}).

\begin{enumerate}
\item $\varphi\circ i_{M}=\varphi$

\item $i_{M}\circ\varphi=\varphi$
\end{enumerate}

\item Composition on left or right preserves order relations in the magnitude
space $E\left(  M\right)  $ (Theorem \ref{composition preserves order})

\begin{enumerate}
\item $\varphi\circ\chi$ has to $\varphi\circ\psi$ the same relation (%
$<$%
, =, or
$>$%
) as $\chi$ has to $\psi$

\item $\chi\circ\varphi$ has to $\psi\circ\varphi$ the same relation (%
$<$%
, =, or
$>$%
) as $\chi$ has to $\psi$
\end{enumerate}
\end{enumerate}

\section{Generalized Multiple\label{sec Generalized Multiple}}

In Section \ref{sec embedding spaces} we showed $H\left(  M,M^{\prime}\right)
$ is itself a magnitude space. The question naturally arises of how $H\left(
M,M^{\prime}\right)  $ might be related to the magnitude spaces $M$ and
$M^{\prime}$. In this section we consider the special case in which $M$ is a
magnitude space with some distinguished element $1$, and $M^{\prime}$ is any
magnitude space such that for each element $a^{\prime}\in M^{\prime} $ there
is an embedding of $M$ into $M^{\prime}$ which maps $1$ into $a^{\prime}$. We
already have two examples of this special case:

\begin{enumerate}
\item $M$ is a well ordered magnitude space and $1$ is the smallest element in
$M$; and $M^{\prime}$ is an arbitrary (Archimedean) magnitude space (Theorem
\ref{embedding discrete magnitude}).

\item $M$ is an arbitrary (Archimedean) magnitude space and $1$ is any element
in $M$; and $M^{\prime}$ is a continuous magnitude space (Theorem
\ref{complete f(a)=b}).
\end{enumerate}

In this section variables $a,b$ are elements of $M$ and $a^{\prime},b^{\prime
}$ are elements of $M^{\prime}$.

\begin{definition}
\label{definition generalized product}Let the mapping $\Psi:M^{\prime
}\rightarrow H\left(  M,M^{\prime}\right)  $ be defined such that for each
$a^{\prime}\in M^{\prime}$, $\Psi a^{\prime}$ is the unique element of
$H\left(  M,M^{\prime}\right)  $ which maps $1$ into $a^{\prime}$. For any
$a\in M$, $a^{\prime}\in M^{\prime}$ we define $aa^{\prime}$ to be $\left(
\Psi a^{\prime}\right)  a$. Note that $aa^{\prime}\in M^{\prime}$.
\end{definition}

\begin{theorem}
The map $\Psi:M^{\prime}\rightarrow H\left(  M,M^{\prime}\right)  $ is an isomorphism.
\end{theorem}

\begin{proof}
Note that $H\left(  M,M^{\prime}\right)  $ is a magnitude space by Theorem
\ref{H(M,N) is a magnitude}. Let $a^{\prime}$ and $b^{\prime}$ be any two
elements of $M^{\prime}$. By assumption there exist elements $\Psi a^{\prime
},\Psi b^{\prime},\Psi\left(  a^{\prime}+b^{\prime}\right)  \in H\left(
M,M^{\prime}\right)  $ which map $1$ into $a^{\prime},b^{\prime},a^{\prime
}+b^{\prime}$ respectively. Now $\Psi a^{\prime}+\Psi b^{\prime}\in H\left(
M,M^{\prime}\right)  $ by Theorem \ref{sum of homomorphisms} and%
\begin{align*}
\left(  \Psi a^{\prime}+\Psi b^{\prime}\right)  1  & =\left(  \Psi a^{\prime
}\right)  1+\left(  \Psi b^{\prime}\right)  1\text{ (Definition
\ref{sum of homomorphisms})}\\
& =a^{\prime}+b^{\prime}\text{. (definition of }\Psi a^{\prime}\text{ and
}\Psi b^{\prime}\text{)}%
\end{align*}
Thus $\Psi\left(  a^{\prime}+b^{\prime}\right)  $ and $\Psi a^{\prime}+\Psi
b^{\prime}$ are each embeddings which map $1$ into $a^{\prime}+b^{\prime}$ and
hence $\Psi\left(  a^{\prime}+b^{\prime}\right)  =\Psi a^{\prime}+\Psi
b^{\prime}$ by Theorem \ref{unique morphism}. Therefore $\Psi$ is an embedding
by Definition \ref{definition homomorphism} and Theorem \ref{morphisms 1-1}.

Now if $\chi\in H\left(  M,M^{\prime}\right)  $, then $\chi$ maps $1$ into
$\chi1$. And also $\Psi\left(  \chi1\right)  \in H\left(  M,M^{\prime}\right)
$ maps $1$ into $\chi1$. Therefore $\chi=\Psi\left(  \chi1\right)  $ by
Theorem \ref{unique morphism}. Hence $\Psi$ is onto and therefore $\Psi$ is an
isomorphism from $M^{\prime}$ onto $H\left(  M,M^{\prime}\right)  $ according
to Definition \ref{definition isomorphism isomorphic}.
\end{proof}

\begin{theorem}
\label{multiple identity element}$1a^{\prime}=a^{\prime}$
\end{theorem}

\begin{proof}
$1a^{\prime}=\left(  \Psi a^{\prime}\right)  1=a^{\prime}$ by Definition
\ref{definition generalized product}.
\end{proof}

The following two theorems were proved separately for integral multiples in
Propositions \ref{V 1} and \ref{V 2}.

\begin{theorem}
\label{multiple V 1}$a\left(  a^{\prime}+b^{\prime}\right)  =aa^{\prime
}+ab^{\prime}$
\end{theorem}

\begin{proof}%
\begin{align*}
a\left(  a^{\prime}+b^{\prime}\right)   & =\left(  \Psi\left(  a^{\prime
}+b^{\prime}\right)  \right)  a\text{ (Definition
\ref{definition generalized product})}\\
& =\left(  \Psi a^{\prime}+\Psi b^{\prime}\right)  a\text{ (Definition
\ref{definition homomorphism})}\\
& =\left(  \Psi a^{\prime}\right)  a+\left(  \Psi b^{\prime}\right)  a\text{
(Definition \ref{definition sum of embeddings})}\\
& =aa^{\prime}+ab^{\prime}\text{ (Definition
\ref{definition generalized product})}%
\end{align*}

\end{proof}

\begin{theorem}
\label{multiple V 2}$\left(  a+b\right)  a^{\prime}=aa^{\prime}+ba^{\prime}$
\end{theorem}

\begin{proof}%
\begin{align*}
\left(  a+b\right)  a^{\prime}  & =\left(  \Psi a^{\prime}\right)  \left(
a+b\right)  \text{ (Definition \ref{definition generalized product})}\\
& =\left(  \Psi a^{\prime}\right)  a+\left(  \Psi a^{\prime}\right)  b\text{
(Definition \ref{definition homomorphism})}\\
& =aa^{\prime}+ba^{\prime}\text{ (Definition
\ref{definition generalized product})}%
\end{align*}

\end{proof}

The following two theorems were proved separately for integral multiples in
Propositions \ref{V 5} and \ref{V 6}.

\begin{theorem}
\label{multiple V 5}If $a^{\prime}>b^{\prime}$, then $aa^{\prime}>ab^{\prime}$
and $aa^{\prime}-ab^{\prime}=a\left(  a^{\prime}-b^{\prime}\right)  $. And,
more generally, $aa^{\prime}$ has to $ab^{\prime}$ the same relation (%
$<$%
, =, or
$>$%
) as $a^{\prime}$ has to $b^{\prime}$.
\end{theorem}

\begin{proof}
Fix $a$ and let $\varphi:M^{\prime}\rightarrow M^{\prime}$ be the function
defined by $\varphi a^{\prime}=aa^{\prime}$. Then $\varphi\left(  a^{\prime
}+b^{\prime}\right)  =\varphi a^{\prime}+$ $\varphi b^{\prime}$ for any
$a^{\prime}$ and $b^{\prime}$ by Theorem \ref{multiple V 1} and hence
$\varphi$ is a homomorphism according to Definition
\ref{definition homomorphism}. Hence Theorems \ref{morphisms monotonic} and
\ref{morphisms 1-1} apply.
\end{proof}

\begin{theorem}
\label{multiple V 6}If $a>b$, then $aa^{\prime}>ba^{\prime}$ and $aa^{\prime
}-ba^{\prime}=(a-b)a^{\prime}$. And, more generally, $aa^{\prime}$ has to
$ba^{\prime}$ the same relation (%
$<$%
, =, or
$>$%
) as $a$ has to $a$.
\end{theorem}

\begin{proof}
Fix $a^{\prime}$ and let $\varphi:M\rightarrow M$ be the function defined by
$\varphi a=aa^{\prime}$. Then $\varphi\left(  a+b\right)  =\varphi a+$
$\varphi b$ for any $a$ and $b$ by Theorem \ref{multiple V 2} and hence
$\varphi$ is a homomorphism according to Definition
\ref{definition homomorphism}. Hence Theorems \ref{morphisms monotonic} and
\ref{morphisms 1-1} apply.
\end{proof}

\section{Generalized Product Operator}

In this section we assume that $M$ is a magnitude space with some
distinguished element $1$ such that for each element $a\in M$ there is an
embedding $\Psi a\in E\left(  M\right)  =H\left(  M,M\right)  $ which maps $1$
into $a$. In other words, we assume the same thing as in the previous section
but, in addition, $M=M^{\prime}$. In this case we will write a product
$ab=\left(  \Psi b\right)  a$ as $a\cdot b$ to emphasize that we have here a
binary operator. Here are two examples of this special case.

\begin{enumerate}
\item $M$ is a well ordered magnitude space and $1$ is the smallest element in
$M$ (for instance, $M=%
\mathbb{N}
$).

\item $M$ is a continuous magnitude space and $1$ is an arbitrary element in
$M$ (for instance, $M=%
\mathbb{R}
_{+}$).
\end{enumerate}

From the preceding section $\Psi:M\rightarrow E\left(  M\right)  $ is an
isomorphism and, since $\Psi$ is an isomorphism,%
\[
\Psi\left(  a+b\right)  =\Psi a+\Psi b
\]
where the addition operator on the right hand side represents the addition of
the endomorphisms $\Psi a$ and $\Psi b$ and the result is another endomorphism.

\begin{theorem}
\label{product identity function}$\Psi1=i_{M}$
\end{theorem}

\begin{proof}
$\Psi1$ is an embedding of $M$ into $M$ which maps $1$ into $1$. And the
identity function $i_{M}$ is also an embedding of $M$ into $M$ which maps
$1$\ into $1$. Therefore $\Psi1=i_{M}$ by Theorem \ref{unique morphism}.
\end{proof}

\begin{theorem}
\label{product vs composition}$\Psi\left(  a\cdot b\right)  =\Psi a\circ\Psi
b$
\end{theorem}

\begin{proof}
First $\left(  \Psi\left(  a\cdot b\right)  \right)  1=a\cdot b$ from the
definition of $\Psi$. Second%
\begin{align*}
\left(  \Psi a\circ\Psi b\right)  1  & =\left(  \Psi b\circ\Psi a\right)
1\text{ (Theorem \ref{endomorphisms commute})}\\
& =\Psi b\left(  \left(  \Psi a\right)  1\right)  \text{ (Definition
\ref{definition composition})}\\
& =\left(  \Psi b\right)  a\text{ (definition of }\Psi\text{)}\\
& =a\cdot b\text{ (Definition \ref{definition generalized product})}%
\end{align*}
and hence also $\left(  \Psi a\circ\Psi b\right)  1=a\cdot b$. Therefore
$\Psi\left(  a\cdot b\right)  =\Psi a\circ\Psi b$ by Theorem
\ref{unique morphism}.
\end{proof}

\begin{remark}
By means of the preceding theorem, we can show that the product binary
operator $\cdot$ on $M$ has analogous properties as the binary operator
$\circ$ on $E\left(  M\right)  $ as summarized at the end of Section
\ref{sec endomorphism spaces}.
\end{remark}

\begin{theorem}
\label{product commutative}$a\cdot b=b\cdot a$
\end{theorem}

\begin{proof}%
\begin{align*}
\Psi\left(  a\cdot b\right)   & =\Psi a\circ\Psi b\text{ (preceding
theorem)}\\
& =\Psi b\circ\Psi a\text{ (Theorem \ref{endomorphisms commute})}\\
& =\Psi\left(  b\cdot a\right)  \text{ (preceding theorem)}%
\end{align*}
Thus $\Psi\left(  a\cdot b\right)  =\Psi\left(  b\cdot a\right)  $ and
therefore $a\cdot b=b\cdot a$ by Theorem \ref{morphisms 1-1}.
\end{proof}

\begin{theorem}
\label{product associative}$\left(  a\cdot b\right)  \cdot c=a\cdot\left(
b\cdot c\right)  $
\end{theorem}

\begin{proof}%
\begin{align*}
\Psi\left(  \left(  a\cdot b\right)  \cdot c\right)   & =\Psi\left(  a\cdot
b\right)  \circ\Psi c\text{ (Theorem \ref{product vs composition})}\\
& =\left(  \Psi a\circ\Psi b\right)  \circ\Psi c\text{ (Theorem
\ref{product vs composition})}\\
& =\Psi a\circ\left(  \Psi b\circ\Psi c\right)  \text{ (Theorem
\ref{composition is associative})}\\
& =\Psi a\circ\Psi\left(  b\cdot c\right)  \text{ (Theorem
\ref{product vs composition})}\\
& =\Psi\left(  a\cdot\left(  b\cdot c\right)  \right)  \text{ (Theorem
\ref{product vs composition})}%
\end{align*}
Thus $\Psi\left(  \left(  a\cdot b\right)  \cdot c\right)  =\Psi\left(
a\cdot\left(  b\cdot c\right)  \right)  $ and therefore $\left(  a\cdot
b\right)  \cdot c=a\cdot\left(  b\cdot c\right)  $ by Theorem
\ref{morphisms 1-1}.
\end{proof}

\section{Symmetric Magnitude Spaces}

From above, we have the definitions of binary product operators for the
natural numbers and the positive real numbers and we have shown the properties
which these two binary operators have in common. There are, of course,
differences between the product operators for the natural and positive real
numbers. In particular, given any $x,y\in%
\mathbb{R}
_{+}$ there is a $q\in%
\mathbb{R}
_{+}$ (called the quotient of $y$ and $x$) such that $y=q\cdot x$. Once again
it is useful to develop this additional property of the real numbers from a
general perspective; for there are also noncontinuous magnitude spaces in
which every pair of elements has a quotient.

\begin{definition}
\label{definition symmetric}A magnitude space $M$ is \textbf{symmetric} if
there is, for every pair $a,b\in M$, an endomorphism of $M$ which maps $a$
into $b$.
\end{definition}

\begin{remark}
Alternatively, a magnitude space $M$ is \textbf{symmetric} if, for every three
elements $a,b,c\in M$, there is a fourth proportional (i.e. an element $d\in
M$ such that $a:b=c:d$).
\end{remark}

\begin{proof}
Theorem \ref{complete f(a)=b}.
\end{proof}

\begin{theorem}
\label{endomorhpism of symmetric is onto}If $M$ is a symmetric magnitude, then
for each $\varphi\in E\left(  M\right)  $, there is a corresponding $\chi\in
E\left(  M\right)  $ such that $\chi\circ\varphi=\varphi\circ\chi=i_{M}$ and
each of $\varphi$ and $\chi$ are onto (and thus automorphisms).
\end{theorem}

\begin{proof}
Let $\varphi\in E\left(  M\right)  $ and pick any element $a\in M$. Since $M$
is a symmetric magnitude space, there is another embedding $\chi\in E\left(
M\right)  $ such that $\chi$ maps $\varphi a$ into $a$. Then $\chi\circ
\varphi$ maps $a$ into $a$. And also $i_{M}$ maps $a$ \ into $a$ by Definition
\ref{definition identity function}. Hence $\chi\circ\varphi=i_{M}$ by Theorem
\ref{unique morphism}. But also $\chi\circ\varphi=\varphi\circ\chi$ by Theorem
\ref{endomorphisms commute} and so $\varphi\circ\chi=i_{M}$. Therefore
$\varphi$ and $\chi$ are each onto by Theorem \ref{proving function is onto}
and so $\varphi$ and $\chi$ are automorphisms according to Definition
\ref{definition endomorphism automorphism}.
\end{proof}

In the remainder of this section $M$ is a symmetric magnitude space with one
element designated by $1$. Then for any $a\in M$, there is a unique embedding
of $M$ into $M$ which maps $1$ into $a$. From the preceding section, we can
define an isomorphism $\Psi:M\rightarrow E\left(  M\right)  $ such that for
$a\in M$, $\Psi a$ is the unique element in $E\left(  M\right)  $ which maps
$1$ into $a$. And, as before, we can define a binary operator on $M$ according
to $a\cdot b=\left(  \Psi b\right)  a$. A property of this binary product
which differs from the product of two natural numbers is given in the
following theorem.

\begin{theorem}
\label{unique quotient}If $a,b\in M$, there exists a unique $d\in M$ such that
$b=d\cdot a$.
\end{theorem}

\begin{proof}
The embedding $\Psi a$ from $M$ into $M$ is an automorphism by the preceding
theorem and hence is one-to-one and onto by Definition
\ref{definition endomorphism automorphism}. Hence there is a unique $d\in M$
such that $\left(  \Psi a\right)  d=b$ by Definitions
\ref{definition one-to-one} and \ref{definition onto}. Therefore there is a
unique $d\in M$ such that $d\cdot a=b $ by Definition
\ref{definition generalized product}.
\end{proof}

\begin{definition}
\label{definition quotient}For $a,b\in M$, we denote by $b/a$ the unique
element $d$ of $M$ such that $b=d\cdot a$.
\end{definition}

\begin{theorem}
$b$ has to $a$ the same relation (%
$<$%
, =, or
$>$%
) as $b/a$ has to $1$.
\end{theorem}

\begin{proof}
$\left(  b/a\right)  \cdot a$ has to $1\cdot a$ the same relation (%
$<$%
, =, or
$>$%
) as $b/a$ has to $1$ by Theorem \ref{multiple V 6}. But $b=\left(
b/a\right)  \cdot a$ by Definition \ref{definition quotient} and $1\cdot a=a$
by Theorem \ref{multiple identity element}. Therefore $b$ has to $a$ the same
relation (%
$<$%
, =, or
$>$%
) as $b/a$ has to $1$.
\end{proof}

\section{Power Functions}

We have mentioned above that it is worthwhile to view the formation of
products in a general way. As a concrete example, let us consider the
definition of $x^{y}$ where $x$ and $y$ are positive real numbers and $x>1$.
In this section $\cdot$ is the product operator in $%
\mathbb{R}
_{+}$.

\begin{definition}
\label{definition R>1}Let $%
\mathbb{R}
_{>1}$ be the elements of $%
\mathbb{R}
_{+}$ which are greater than $1$.
\end{definition}

\begin{theorem}
If $x,y\in%
\mathbb{R}
_{>1}$, then $x\cdot y\in%
\mathbb{R}
_{>1}$.
\end{theorem}

\begin{proof}
If $x,y\in%
\mathbb{R}
_{>1}$, then $x>1$ and $y>1$ by Definition \ref{definition R>1}. And $x\cdot
y>1\cdot y$ by Theorem \ref{multiple V 6} and $1\cdot y=y$ by Theorem
\ref{multiple identity element}. Thus $x\cdot y>y$ and $y>1$. Therefore
$x\cdot y>1$ by Theorem \ref{Transitivity of >} and $x\cdot y\in%
\mathbb{R}
_{>1}$ by Definition \ref{definition R>1}.
\end{proof}

\begin{theorem}
$%
\mathbb{R}
_{>1}$ with the multiplicative operator $\cdot$ is a magnitude space.
\end{theorem}

\begin{proof}
The multiplicative operator $\cdot$ is a binary operator on $%
\mathbb{R}
_{>1}$ by the preceding theorem. And $\cdot$ is associative and commutative by
Theorems \ref{product associative} and \ref{product commutative}.

It remains to show that the binary operator $\cdot$ on $%
\mathbb{R}
_{>1}$ is trichotomous. Let $x,y\in%
\mathbb{R}
_{>1}$. Note that $y=\left(  y/x\right)  \cdot x=x\cdot\left(  y/x\right)  $
and $x=\left(  x/y\right)  \cdot y=y\cdot\left(  x/y\right)  $ by Definition
\ref{definition quotient} and Theorem \ref{product commutative}. From
trichotomy in $%
\mathbb{R}
_{+}$, exactly one of the following is true: $y<x$, or $y=x$, or $y>x$.

Now in each case, there is a unique $d\in%
\mathbb{R}
_{+}$, namely $d=x/y$, such that $x=y\cdot d$ by Theorem \ref{unique quotient}%
. But $d=x/y\in%
\mathbb{R}
_{>1}$ only if $y<x$. Similarly, in each case, there is a unique $d\in%
\mathbb{R}
_{+}$, namely $d=y/x$ such that $y=x\cdot d$. But $d=y/x\in%
\mathbb{R}
_{>1}$ only if $y>x$. Thus there are three mutually exclusive cases: $x=y\cdot
d$ for some $d\in%
\mathbb{R}
_{>1}$, or $x=y$, or $y=x\cdot d$ for some $d\in%
\mathbb{R}
_{>1}$. Thus $%
\mathbb{R}
_{>1}$ with the $\cdot$ binary operator has the trichotomy property in
Definition \ref{definition magnitude}.
\end{proof}

\begin{theorem}
$%
\mathbb{R}
_{>1}$ with the multiplicative operator $\cdot$ is a continuous magnitude space.
\end{theorem}

\begin{proof}
We begin with an elementary observation. If $A$ is a nonempty set with some
element $a$ greater than $1$ and $b$ is an upper bound of $A$, then $b>1$. For
if $b$ is an upper bound of $A$, then $a\leq b$ by Definition
\ref{definition bounds}. And from $1<a$ and $a\leq b$ follows $1<b$ by Theorem
\ref{Transitivity mixed}.

Definition \ref{definition <} defines $<$ and $>$ in a magnitude space in
terms of the binary operator of the magnitude space. In the proof of the
preceding theorem, it may be observed that for $x,y\in%
\mathbb{R}
_{>1}$, $x$ has to $y$ the same relation (%
$<$%
, =, or
$>$%
) in the order defined by $+$ binary operator as $x$ has to $y$ in the order
defined defined by the $\cdot$ operator.

Now let $A$ be a nonempty subset of $%
\mathbb{R}
_{>1}$ with a nonempty set $B$ of upper bounds with respect to the order
defined by the $\cdot$ operator. From the preliminary observation, $B$
coincides with the set of upper bounds of $A$ with respect to the order
defined by the $+$ operator. But $%
\mathbb{R}
_{+}$ is continuous by Definition \ref{definition real number} and hence $B$
has a smallest element with respect to the order defined by the $+$ operator
by Definition \ref{definition continuous}. Therefore $B$ has a smallest
element with respect to the order defined by the $\cdot$ operator. Therefore $%
\mathbb{R}
_{>1}$ is continuous by Definition \ref{definition continuous}.
\end{proof}

\ 

Up to the preceding theorem, we have consistently used the symbol $+$ for the
binary operator in a magnitude space. To be precise, if we have two magnitude
spaces $M$ and $M^{\prime}$, the binary operators are in general not the same.
If $a,b\in M$ and $a^{\prime},b^{\prime}\in M^{\prime}$ then the $+$ sign in
the expression $a+b$ is understood to be the binary operator in $M$ and the
$+$ sign in the expression $a^{\prime}+b^{\prime}$ is understood to be the
binary operator in $M^{\prime}$. Admittedly it would be more precise to denote
the binary operator in $M^{\prime}$ by $+^{\prime}$ and to write $a^{\prime
}+^{\prime}b^{\prime}$ but we have left it up to reader to make this
distinction. In particular we defined a map $\varphi:M\rightarrow M^{\prime}$
to be a homomorphism if for all $a,b\in M$, $\varphi\left(  a+b\right)
=\left(  \varphi a\right)  +\left(  \varphi b\right)  $ and the first $+$ sign
refers to the binary operator in $M $ while the second $+$ sign refers to the
binary operator in $M^{\prime}$. And if we had a different symbol for the
binary operator in $M^{\prime}$, say $\times$, we would of course say that
$\varphi:M\rightarrow M^{\prime} $ is an homomorphism if for all $a,b\in M$,
$\varphi\left(  a+b\right)  =\left(  \varphi a\right)  \times\left(  \varphi
b\right)  $. In the present case, we denote the binary operator in the
magnitude space $%
\mathbb{R}
_{>1}$ by $\cdot$ and so by a homomorphism $%
\mathbb{R}
_{+}$ into $%
\mathbb{R}
_{>1}$ we mean a function $\varphi:%
\mathbb{R}
_{+}\rightarrow%
\mathbb{R}
_{>1}$ such that for all $x,y\in%
\mathbb{R}
_{+}$, $\varphi\left(  x+y\right)  =\left(  \varphi x\right)  \cdot\left(
\varphi y\right)  $.

Having shown that $%
\mathbb{R}
_{>1}$ is a continuous magnitude space, we know that for each $y\in%
\mathbb{R}
_{>1}$ there exists a unique embedding of $%
\mathbb{R}
_{+}$ into $%
\mathbb{R}
_{>1}$ which maps $1$ into $y$ by Theorem \ref{complete f(a)=b}. The
assumptions of Section \ref{sec Generalized Multiple} are satisfied with $M=%
\mathbb{R}
_{+}$ and $M^{\prime}=%
\mathbb{R}
_{>1}$. And, in accordance with Section \ref{sec Generalized Multiple}, for
each $x\in%
\mathbb{R}
_{>1}$ we define $\Psi x$ to be the unique embedding of $%
\mathbb{R}
_{+}$ into $%
\mathbb{R}
_{>1}$ which maps $1$ into $x$. For $y\in%
\mathbb{R}
_{+}$ and $x\in%
\mathbb{R}
_{>1}$ we then have, as before, a product $yx=\left(  \Psi x\right)  y\in%
\mathbb{R}
_{>1}$. To avoid confusing this definition of $yx$ with multiplication in $%
\mathbb{R}
_{+}$ we make the following definition.

\begin{definition}
If $x\in%
\mathbb{R}
_{>1}$ then we denote by $\Psi x$ the unique embedding of $%
\mathbb{R}
_{+}\rightarrow%
\mathbb{R}
_{>1}$ which maps $1$ into $x$. And for $y\in%
\mathbb{R}
_{+}$, we denote $\left(  \Psi x\right)  y$ by $x^{y}$.
\end{definition}

\begin{theorem}
$\left(  x_{1}\cdot x_{2}\right)  ^{y}=x_{1}^{y}\cdot x_{2}^{y}$
\end{theorem}

\begin{proof}
Theorem \ref{multiple V 1}.
\end{proof}

\begin{theorem}
$x^{y_{1}+y_{2}}=x^{y_{1}}\cdot x^{y_{2}}$
\end{theorem}

\begin{proof}
Theorem \ref{multiple V 2}.
\end{proof}

\end{document}